\documentclass[10pt]{amsart}

\usepackage{graphicx}

\usepackage[USenglish]{babel}
\usepackage{latexsym}
\usepackage{amsmath}
\usepackage{amsfonts}
\usepackage{amssymb}
\usepackage{esint}
\usepackage[all]{xy}

\renewcommand{\d}{\delta }
\newcommand{\D }{\Delta }

\newcommand{\M}{\mathcal{M}}

\newcommand{\e }{\varepsilon }
\newcommand{\g }{\gamma}

\renewcommand{\l }{\lambda }
\renewcommand{\L }{\Lambda }

\newcommand{\n }{\nabla }
\newcommand{\var }{\varphi }
\newcommand{\rh }{\rho }

\newcommand{\s }{\sigma }
\newcommand{\Sg }{\Sigma}

\renewcommand{\th }{\theta }
\renewcommand{\o }{\omega }
\renewcommand{\O }{\Omega }

\newcommand{\ov}{\overline}

\def\o{\Omega}
\def\p{\partial}
\def\le{\langle}
\def\r{\rangle}
\newcommand{\wtilde }{\widetilde}

\newcommand{\be}{\begin{equation}}
\newcommand{\ee}{\end{equation}}

\newtheorem{corollary}{Corollary}[section]
\newtheorem{remark}{Remark}[section]
\newcommand{\R}{\mathbb{R}}

\newcommand{\N}{\mathbb{N}}
\newcommand{\no}{\noindent}
\newcommand{\dis}{\displaystyle}

\newcommand{\dkr}{{\bf d}}

\newtheorem{theorem}{Theorem}[section]
\newtheorem{proposition}[theorem]{Proposition}
\newtheorem{example}[theorem]{Example}

\newcommand{\bpr}{\begin{proposition}}
\newcommand{\epr}{\end{proposition}}
\newcommand{\bex}{\begin{example}\rm}
\newcommand{\eex}{\end{example}}

\begin{document}

\newtheorem{lem}{Lemma}[section]
\newtheorem{pro}[lem]{Proposition}
\newtheorem{thm}[lem]{Theorem}
\newtheorem{rem}[lem]{Remark}
\newtheorem{cor}[lem]{Corollary}
\newtheorem{df}[lem]{Definition}

\title[The Tzitz\'eica equation]{Analytic aspects of the Tzitz\'eica equation: blow-up analysis and existence results}

\author{Aleks Jevnikar, Wen Yang}

\address{ Aleks Jevnikar,~University of Rome `Tor Vergata', Via della Ricerca Scientifica 1, 00133 Roma, Italy}
\email{jevnikar@mat.uniroma2.it}

\address{ Wen ~Yang,~Center for Advanced Study in Theoretical Sciences (CASTS), National Taiwan University, Taipei 10617, Taiwan}
\email{math.yangwen@gmail.com}

\thanks{We thank Professor Juncheng Wei for suggesting us the problem and Francesca De Marchis for pointing out Remark~\ref{rem:analog}.}



\begin{abstract}
We are concerned with the following class of equations with exponential nonlinearities:
$$
	\Delta u+h_1e^u-h_2e^{-2u}=0 \qquad \mathrm{in}~B_1\subset\R^2,
$$
which is related to the Tzitz\'eica equation. Here $h_1, h_2$ are two smooth positive
functions. The purpose of the paper is to initiate the analytical study of the above equation and to give a quite complete picture both for what concerns the blow-up phenomena and the existence issue.

\medskip

In the first part of the paper we provide a quantization of local blow-up masses associated to a blowing-up sequence of solutions. Next we exclude the presence of blow-up points on the boundary under the Dirichlet boundary conditions.

\medskip

In the second part of the paper we consider the Tzitz\'eica equation on compact surfaces: we start by proving a sharp Moser-Trudinger inequality related to this problem. Finally, we give a general existence result.
\end{abstract}

\maketitle

\section{Introduction}

\medskip

We consider here the following equation
\begin{equation} \label{eq1}
	\Delta u+h_1e^u-h_2e^{-2u}=0 \qquad \mathrm{in}~B_1\subset\R^2,
\end{equation}
where $h_1, h_2$ are two smooth positive functions and $B_1$ is the unit ball in $\R^2$. Equation \eqref{eq1} is related to the Tzitz\'eica equation which arises in differential geometry in the context of surfaces with constant affine curvature, see \cite{tz, tz2}. Moreover, it appears also in several other frameworks: it is related to the Euler's equation for one-dimensional ideal gas dynamics \cite{ga,ga2,sc-ro,tz3}, while in magnetohydrodynamics it is in correspondence to the Hirota-Satsuma PDE \cite{hi-ra, hi-sa}; see also \cite{cmg,du-pla} and the reference therein. We point out that the case of the nonlinearty of the form $e^{\gamma u}$, $\gamma\in[-1,1]$, was recently considered in \cite{pi-ri,ri-ta,rtzz}. Our analysis extends some results obtained for the latter problem, see Remarks \ref{rem:boundary}, \ref{rem:analog}.

The aim of this paper is to start the analysis concerning the equation \eqref{eq1} and to provide detailed blow-up information as well as the general existence results.

\medskip

When $h_2 \equiv 0$ in \eqref{eq1} we obtain the well-known Liouville equation
\begin{equation} \label{liouv}
\Delta u + h \,e^u=0 \qquad \mathrm{in}~B_1\subset\R^2.
\end{equation}
The latter equation is related to the change of Gaussian curvature under conformal deformation of the metric, see \cite{ba, cha, cha2, ly, s}. On the other hand, in mathematical physics it is a model for the mean field of Euler flows, see \cite{clmp} and \cite{ki}. This equation has become quite standard now and we refer the interested reader to \cite{mal} and \cite{tar}.

\medskip

As many geometric problems, \eqref{eq1} (and \eqref{liouv}) presents loss of compactness phenomena, as its solutions might blow-up. Concerning \eqref{liouv} it is a well-known fact (see \cite{bm, li, li-sha}) that for a sequence of blow-up solutions $u_k$ to \eqref{liouv} with blow-up point $\bar x$ there is a quantized \emph{local mass}, more precisely:
\begin{equation} \label{quant}
	\lim_{\d \to 0} \lim_{k \to \infty} \int_{B_\d (\bar x)} h^k e^{u_k} = 8\pi.
\end{equation}

\medskip

On the other hand, when $h_2\neq0$ there are no results concerning the blow-up behavior of solutions to equation \eqref{eq1}. A similar problem, namely the following sinh-Gordon equation
\begin{equation}\label{sinh-gordon}
	\Delta u+h_1e^u-h_2e^{-u}=0 \qquad \mathrm{in}~B_1\subset\R^2,
\end{equation}
was considered by \cite{os, os1} and \cite{jwyz} under the assumption that $h_1=h_2$. Later, in \cite{jwy} the case of general $h_1,h_2$ was studied and the authors proved an analogous quantization property as the one in \eqref{quant}, namely that the blow-up limits are multiple of $8\pi$. The latter blow-up situation may indeed occur, see \cite{ew} and \cite{gp}.

\medskip

The first goal of this paper is to extend this studies to the Tzitz\'eica equation \eqref{eq1} and to prove a quantization result. To this end we give the following preparations.

Let $u_k$ be a sequence of blow-up solutions
\begin{equation}
\label{seq}
\Delta u_k+h_1^ke^{u_k}-h_2^ke^{-2u_k}=0,
\end{equation}
such that $0$ is the only blow-up point in $B_1$, more precisely:
\begin{equation}
\label{a1}
\max_{K\subset\subset B_1\setminus\{0\}}|u_k|\leq C(K),\quad \max_{x\in B_1}\{|u_k(x)|\}\rightarrow\infty.
\end{equation}
We will call $\int_{B_1}h_1^ke^{u_k}$ the energy of $u_k$. Furthermore, we suppose that
\begin{equation}
\label{a2}
h_1^k(0)=h_2^k(0)=1, ~ \frac1C\leq h_i^k(x)\leq C, ~ \|h_i^k(x)\|_{C^3(B_1)}\leq C, \qquad \forall x\in B_1, ~i=1,2,
\end{equation}
for some constant $C>0$. A natural assumption is the following:
\begin{align}
\label{a3}
\begin{split}
|u_k(x)-u_k(y)|\leq C,\qquad \forall~x,y\in\partial B_1, \\
\int_{B_1}h_1^ke^{u_k}\leq C, \qquad \int_{B_1}h_2^ke^{-2u_k}\leq C,
\end{split}
\end{align}
where $C$ is independent of $k.$
\medskip

The local masses are defined as
\begin{align}
\label{localmass}
\begin{split}
&\sigma_1=\lim_{\delta\rightarrow0}\lim_{k\rightarrow\infty}\frac{1}{2\pi}\int_{B_{\delta}}h_1^ke^{u_k}, \\
&\sigma_2=\lim_{\delta\rightarrow0}\lim_{k\rightarrow\infty}\frac{1}{2\pi}\int_{B_{\delta}}h_2^ke^{-2u_k}.
\end{split}
\end{align}

\medskip

Our first result is the following quantization property:
\begin{theorem}
\label{th:quantization}
Let $\sigma_i$ be defined as in (\ref{localmass}). Suppose $u_k$ satisfies (\ref{seq}), (\ref{a1}), (\ref{a3}) and $h_i^k$ satisfy (\ref{a2}). Then $\sigma_1\in 4\N$ and $\sigma_2\in 2\N$.
\end{theorem}

\begin{remark}
Actually, the possible values of $\sigma_i$ are more restrictive. In fact, they have to satisfy the following relation:
$$
4\left(\sigma_1+\frac{\sigma_2}{2}\right)=(\sigma_1-\sigma_2)^2.
$$
Therefore, $(\sigma_1,\sigma_2)$ is one of the two following types:
\begin{align*}
\bigr(2m(3m-1),2(3m-1)(m-1)\bigr) \quad \mathrm{or} \quad \bigr(2(3m-2)(m-1),2(3m-5)(m-1)\bigr),
\end{align*}
for some $m\in \mathbb{Z}$, excluding the case $(0,0)$.
\end{remark}

\medskip

The strategy to prove Theorem \ref{th:quantization} goes as follows: we start by constructing a selection process to describe the situation around the blow-up point. This idea was originally introduced in \cite{chen-lin,ly,s2} for prescribed curvature problem and then adapted in \cite{lin-wei-zh} to treat the $SU(3)$ Toda system. Recently, an analogous method was used in \cite{jwy} to attack the sinh-Gordon case \eqref{sinh-gordon}. In this procedure we detect a finite number of \emph{blowing-up disks} where the local energy is related to that of globally defined Liouville equations. Then one can prove that in each disk the local mass of at least one of the two components $u_k$ and $-2u_k$ is quantized. We then combine disks which are close to each other (we will collect them in \emph{groups}) and get that the local mass of at least one of the two components is still quantized. Finally, we apply an energy quantization result in \cite{lin-wei-zh2} concerning global Liouville equation with singularities jointly with a Pohozaev identity to deduce Theorem \ref{th:quantization}.

\

We next continue the blow-up analysis by considering the problem \eqref{eq1} with Dirichlet boundary conditions, more precisely we are concerned with
\begin{equation} \label{eq2}
\begin{split}
\left\{
\begin{array}{ll}
	\Delta u+h_1e^u-h_2e^{-2u}=0 & \mathrm{in}~\Omega, \\
	u=0 & \mathrm{on}~\partial \Omega,
\end{array}	
\right.
\end{split}	
\end{equation}
where $\Omega$ is a bounded domain in $\R^2$ with smooth boundary $\partial\Omega$. Our aim is to show that there are no blow-up points at the boundary $\partial\Omega$. To this end, we will follow the argument presented in \cite{wei-ya}, where the sinh-Gordon case \eqref{sinh-gordon} is studied. The argument was originally introduced in \cite{lwy,ro-wei} in treating a fourth order mean field equation and the $SU(3)$ Toda system, respectively. Recently, this strategy was refined in \cite{wei-ya} to attack the sinh-Gordon problem \eqref{sinh-gordon}.

The main difficulty is due to the non-trivial blow-up behavior of a sequence of solutions to \eqref{eq2}. More precisely, the work of \cite{apr} suggests the existence of blowing-up sequences of solutions to \eqref{eq2} with no concentration property. In other words, there is a mass residual and the bubbling solutions may not converge to a summation of Green functions away from the blow-up points. This is in contrast to the standard Liouville equation \eqref{liouv}, see for example \cite{bm,li,li-sha}. This leads to refine the blow-up analysis and by means of a Pohozaev identity we are still able to prove the following result.
\begin{theorem} \label{th:boundary}
Let $u_k$ be a blowing-up sequence of solutions to \eqref{eq2} and suppose that the analogous conditions as in \eqref{a2}, \eqref{a3} hold true. More precisely we have
$$
	\sup_{x\in\Omega} |u_k(x)|\to+\infty.
$$
Then the blow-up set $S$ of $|u_k|$ is finite and it holds $S\cap\partial\Omega = \emptyset$.
\end{theorem}

\begin{remark} \label{rem:boundary}
The argument applies to a more general classes of problems of the form
$$
\left\{
\begin{array}{ll}
	\Delta u+h_1e^u-h_2e^{-a u}=0 & \mathrm{in}~\Omega, \\
	u=0 & \mathrm{on}~\partial \Omega,
\end{array}	
\right.
$$
with $a>0$ and the same conclusions as in Theorem \ref{th:boundary} hold for the above problem.
\end{remark}

\

In the second part of the paper we consider the Tzitz\'eica equation on a compact surface $M$:
\begin{equation} \label{eq3}
  - \D u = \rho_1 \left( \frac{h_1 e^{u}}{\int_M
      h_1 e^{u} \,dV_g} - \frac{1}{|M|} \right) - \rho_2 \left( \frac{h_2 e^{-2u}}{\int_M
      h_2 e^{-2u} \,dV_g} - \frac{1}{|M|} \right),
\end{equation}
where $\D=\D_g$ is the Laplace-Beltrami operator, $h_1, h_2$ are smooth positive functions, $\rho_1, \rho_2$ are two positive parameters and $M$ is a compact orientable surface with no boundary and a Riemannian metric $g$. For the sake of simplicity, we normalize the total volume of $M$ so that $|M|=1$.

The purpose here is to give a general (for a general choice of the parameters $\rho_1, \rho_2$ and the manifold $M$) existence result for equation \eqref{eq3} by using variational methods.

\medskip

Problem \eqref{eq3} has a variational structure and the associated energy functional $J_\rho : H^1(M)\to\R$, $\rho=(\rho_1,\rho_2)$ is given by
\begin{equation} \label{functional}
	J_\rho(u) = \frac{1}{2}\int_M |\nabla u|^2 \,dV_g - \rho_1 \left( \log\int_M h_1  e^u \,dV_g  - \int_M u \,dV_g \right)
           -  \frac{\rho_2}{2} \left( \log\int_M h_2 e^{-2u} \,dV_g + \int_M 2u \,dV_g \right).
\end{equation}

\medskip

We recall that the one-parameter case ($\rho_2=0$) corresponds to the standard mean field equation
\begin{equation} \label{liouv2}
  - \D u = \rho \left( \frac{h\,e^{u}}{\int_M
      h\, e^{u} \,dV_g} - \frac{1}{|M|} \right)
\end{equation}
and we refer to the reviews \cite{mal,tar}. In this framework the associated energy functional is in the form
\begin{equation} \label{func2}
	I_\rho(u) = \frac{1}{2}\int_M |\nabla u|^2 \,dV_g - \rho  \left( \log\int_M h\,  e^u \,dV_g  - \int_M u \,dV_g \right).
\end{equation}
The first key tool in treating functional like \eqref{func2} (and \eqref{functional}) is the Moser-Trudinger inequality, which is stated as follows:
\begin{equation}\label{ineq}
	8\pi \log\int_M e^{u-\ov u} \, dV_g \leq \frac 12 \int_M |\n u|^2\,dV_g + C_{M,g}, \qquad \ov u= \fint_M u\,dV_g.
\end{equation}
Inequality \eqref{ineq} implies that the functional \eqref{func2} is bounded from below and coercive if $\rho<8\pi$ and a solution to the Liouville equation \eqref{liouv2} is obtained by direct minimization. On the other hand, when $\rho>8\pi$ the functional $I_\rho$ is unbounded from below and one usually seeks for critical point of saddle type. To handle the problem we need improved Moser-Trudinger inequalities: roughly speaking, the more the term $e^u$ is \emph{spread} over the surface the bigger is the constant in the left-hand side of \eqref{ineq}, see \cite{ch-li}. As a consequence we get a better lower bound on the functional $I_\rho$. One can interpret this in the following way: if $\rho<8(k+1)\pi, k\in\N$ and $I_\rho(u)$ is large negative, $e^u$ is forced to accumulate near at most $k$ points of $M$. To describe these configurations one is leaded to introduce the $k$-th \emph{formal barycentres} of $M$
\begin{equation}	\label{M_k}
	M_k = \left\{ \sum_{i=1}^k t_i \d_{x_i} \, : \, \sum_{i=1}^k t_i=1, \,  x_i\in M \right\}.
\end{equation}
It is indeed possible to show that the low sublevels of the functional $I_\rho$ have at least the topology of $M_k$, which is non-trivial. One can then run a min-max scheme based on this fact to prove existence of solutions to \eqref{liouv2} for $\rho\notin 8\pi\N$.

\medskip

Before discussing our problem let us report what is known about the sinh-Gordon case \eqref{sinh-gordon} since we will proceed here in the same spirit. First of all, an analogous inequality as the one in \eqref{ineq} was proven to hold also for \eqref{sinh-gordon} in \cite{os}. The rough picture of the known results is then the following and is mainly based on variational techniques: when (at least one of) the parameters $\rho_i$ are small then one can exploit some \emph{partial} coerciveness of the related energy functional to exploit the analysis developed for the Liouville equation \eqref{liouv2} (see \cite{mal}) and get a solution to \eqref{sinh-gordon}, see \cite{zhou}.

The first existence result in a non-coercive regime was given in \cite{jev} via a detailed description of the sublevels of the associated energy functional. Later, the authors in \cite{bjmr} provided a general existence result under the assumption the surface has positive genus. See also \cite{jev2, jwy2} for a different approach that relies on the topological degree theory.

\medskip

Concerning the Tzitz\'eica equation \eqref{eq3} and the related functional \eqref{functional}, our first result is to prove a similar sharp Moser-Trudinger inequality as \eqref{ineq}. This is carried out by a similar argument as in \cite{os}, following the idea in \cite{ding} (which was used also for the Toda system in \cite{bat-mal, jost-wang}) jointly with the quantization result stated in Theorem~\ref{th:quantization}.
\begin{theorem}
\label{th:m-t}
The functional $J_\rho$ \eqref{functional} is bounded from below if and only if $\rho_1\leq8\pi$ and $\rho_2\leq4\pi$. Namely, there exists a constant $C_{M,g}$ such that for any $u\in H^1(M)$ it holds
\begin{equation} \label{m-t}
8\pi \log\int_M e^{u-\ov u} \, dV_g + \frac{4\pi}{2} \log\int_M e^{-2(u-\ov u)} \, dV_g \leq \frac 12 \int_M |\n u|^2\,dV_g + C_{M,g},\qquad \ov u= \fint_M u\,dV_g.
\end{equation}
\end{theorem}

\medskip

The latter inequality can be also obtained by reinterpreting a result concerning mean field equations involving probability measures in \cite{ri-su}. By the above result we get coerciveness of the functional \eqref{functional} for $\rho_1<8\pi$, $\rho_2<4\pi$ and hence by direct minimization we deduce the following result.
\begin{corollary}
Let $\rho_1<8\pi$ and $\rho_2<4\pi$. Then, the problem \eqref{eq3} admits a solution.
\end{corollary}

\medskip

On the other hand, when both parameter are large, namely $\rho_1>8\pi$ and $\rho_2>4\pi$, the problem becomes more involved due to the interaction of $e^u$ and $e^{-2u}$. Nevertheless, we are able to prove the following general existence result.
\begin{theorem}
\label{th:existence}
Let $M$ be a compact surface with positive genus and suppose that $\rho_1\notin8\pi\N$, $\rho_2\notin 4\pi\N$. Then, \eqref{eq3} has a solution.
\end{theorem}

\begin{rem}
Actually, using Morse theory we can also get a multiplicity result. Indeed, suppose $\rho_1 \in (8k\pi, 8(k + 1)\pi)$ and $\rho_2 \in (4l\pi, 4(l + 1)\pi)$, $k, l \in \N$ and let $g(M)>0$ be the genus of $M$. Then, for a generic choice of the metric $g$ and of the functions $h_1,h_2$ it holds
$$
 \#\bigr\{ \mbox{solutions of \eqref{eq3}} \bigr\} \geq \binom{k+g(M)-1}{g(M)-1} \binom{l+g(M)-1}{g(M)-1}.
$$
One can follow the argument in \cite{jev3} to deduce the latter estimate. To keep the paper shorter we will present just the existence result stated in Theorem \ref{th:existence}.
\end{rem}

\medskip

We follow here the argument introduced in \cite{bjmr} for the Toda system. Similarly as for the Liouville case \eqref{liouv2} one can use improved Moser-Trudinger inequalities to show that if $\rho_1<8k\pi$, $\rho_2<4l\pi$, $k,l\in\N$, and if $J_\rho(u)$ is large negative, then either $e^u$ (normalized) is close to $M_k$ or $e^{-2u}$ is close to $M_l$ in the distributional sense, see the definition in \eqref{M_k}. This alternative can be expressed by using the \emph{topological join} of $M_k$ and $M_l$. We recall that, the topological join of two topological sets is given by
\begin{equation}\label{join}
 A*B = \frac{ \Bigr\{ (a,b,s): \; a \in A,\; b \in B,\; s \in [0,1]  \Bigr\}}E,
\end{equation}
where $E$ is the following equivalence relation:
$$
(a_1, b,1) \stackrel{E}{\sim} (a_2,b, 1)  \quad \forall a_1, a_2
\in A, b \in B
  \qquad \quad \hbox{and} \qquad \quad
(a, b_1,0)  \stackrel{E}{\sim} (a, b_2,0) \quad \forall a \in A,
b_1, b_2 \in B.
$$
In this way we can map the low sublevels of $J_\rho$ into $M_k * M_l$: the idea is that the
join parameter $s$ expresses whether $e^{u}$ is closer to $M_k$ or $e^{-2u}$ is closer to $M_l$.

Next, we exploit the fact the genus of $M$ is positive to construct two disjoint simple non-contractible curves  $\gamma_1, \gamma_2$ such that $M$ retracts on each of them through continuous maps $\Pi_1, \Pi_2$, respectively. Via these retractions we restrict our target from $M_k * M_l$ to $(\gamma_1)_k *(\gamma_2)_l$ only. On the other way round, we can construct test functions modeled on $(\gamma_1)_k * (\gamma_2)_l$.

We can apply then a topological argument based on this analysis to get a solution to problem \eqref{eq3}. In order to run the topological argument one needs some compactness property: we will exploit the compactness of the set of solutions to \eqref{eq3}. This property is stated in Corollary \ref{cmpt} and it is deduced from our quantization result in Theorem \ref{th:quantization}. It is in this step that we have to suppose $\rho_1\notin8\pi\N$ and $\rho_2\notin 4\pi\N$.

\

\begin{remark} \label{rem:analog}
We point out that by a simple substitution we obtain analogous blow-up and existence results (with obvious modifications) for the following equation:
$$
	\Delta v+\wtilde h_1e^v-\wtilde h_2e^{-\frac 12 v}=0 \qquad \mathrm{in}~B_1\subset\R^2.
$$
\end{remark}

\medskip

The organization of this paper is as follows. In Section \ref{sec:quantization} we study the blow-up phenomenon related to equation \eqref{eq1} on bounded domains and we prove the quantization of the local masses stated in Theorem \ref{th:quantization}. In Section \ref{sec:boundary} we proceed further and under Dirichlet boundary condition we show that problem \eqref{eq2} has no blow-up at the boundary, see Theorem \ref{th:boundary}. Next, in Section \ref{sec:m-t} we move to the equation \eqref{eq3} defined on a compact surface and we get the related sharp Moser-Trudinger inequality stated in Theorem \ref{th:m-t}. Finally, in Section \ref{sec:existence} we introduce the variational argument to prove the general existence result, namely Theorem \ref{th:existence}.

\

\begin{center}
\textbf{Notation}
\end{center}

\

The symbol $B_r(p)$ stands for the open metric ball of
radius $r$ and center $p$. To simplify the notation we will write $B_r$ for balls which are centered at $0$. We will use the notation $a\sim b$ for two comparable quantities $a$ and $b$.

The average of $u$ will be indicated by $\ov u$. The sublevels of the functional $J_\rho$ will be denoted by
\begin{equation} \label{sub}
	J_\rho^a = \bigr\{ u\in H^1(M) \,:\, J_\rho(u) \leq a  \bigr\}.
\end{equation}
Denoting by $\mathcal M(M)$ the  set of all Radon measures on $M$ we consider the Kantorovich-Rubinstein distance on it:
\begin{equation} \label{dist}
	\dkr (\mu_1,\mu_2) = \sup_{\|f \|_{Lip}\leq 1} \left| \int_M f \, d\mu_1 - \int_M f\,d\mu_2 \right|, \qquad \mu_1,\mu_2 \in \mathcal M(M).
\end{equation}

\medskip

Throughout the paper the letter $C$ will stand for large constants which
are allowed to vary among different formulas or even within the same lines.
When we want to stress the dependence of the constants on some
parameter (or parameters), we add subscripts to $C$, as $C_\d$,
etc. We will write $o_{\alpha}(1)$ to denote
quantities that tend to $0$ as $\alpha \to 0$ or $\alpha \to
+\infty$; we will similarly use the symbol
$O_\alpha(1)$ for bounded quantities.

\

\section{Classification of the blow-up limits} \label{sec:quantization}

\medskip
\subsection{Some useful tools.}
In this subsection we list some lemmas which will be used in the proof of the quantization resul of Theorem \ref{th:quantization}. The proof of all the results can be found in \cite{jwy} with minor modifications. We start by the following selection process of the bubbling disks.
\begin{proposition}
\label{pr2.1}
Let $u_k$ be a sequence of blow-up solutions of \eqref{seq} that satisfy \eqref{a1} and \eqref{a3}, and suppose that $h_i^k$ satisfy \eqref{a2}. Then there exists finite sequence of points $\Sigma_k=\{x_1^k,\cdots,x_m^k\}$ (all $x_j^k\rightarrow0,~j=1,\cdots,m$) and positive $l_1^k,\cdots,l_m^k\rightarrow0$ such that, letting $M_{k,j}=\max\{u_k(x_j^k),-2u_k(x_j^k)\}$, we have
\begin{enumerate}
  \item $M_{k,j}=\max_{B_{l_j^k}(x_j^k)}\{|u_k|\}$~for~$j=1,\cdots,m.$
  \item $\exp\left(\frac12 M_{k,j}\right)l_j^k\rightarrow\infty$ for $j=1,\cdots,m.$
  \item Let $\varepsilon_{k,j}=e^{-\frac{1}{2}M_{k,j}}$. In each $B_{l_j^k}(x_j^k)$ we define the dilated functions
  \begin{align}
  \label{2.1}
  \begin{split}
  &v_1^k(y)=u_k\bigr(\varepsilon_{k,j}y+x_j^k\bigr)+2\log\varepsilon_{k,j}, \\
  &v_2^k(y)=-2u_k\bigr(\varepsilon_{k,j}y+x_j^k\bigr)+2\log\varepsilon_{k,j}.
  \end{split}
  \end{align}
  Then it holds that either $v_1^k$ converges to a function $v_1$ in $C_{loc}^2(\R^2)$ which satisfies the equation $\Delta v_1+e^{v_1}=0$ and $v_2^k$ tends to minus infinity over all compact subsets of $\R^2$ or $v_2^k$ converges to a function $v_2$ in $C_{loc}^2(\R^2)$ which satisfies the equation $\Delta v_2+2e^{v_2}=0$ and $v_1^k$ tends to minus infinity over all compact subsets of $\R^2.$
  \item There exists a constant $C_1>0$ independent of $k$ such that
  \begin{align*}
  \max\{u_k(x),-2u_k(x)\}+2\log\mathrm{dist}(x,\Sigma_k)\leq C_1, \qquad \forall x\in B_1.
  \end{align*}
\end{enumerate}
\end{proposition}

\medskip

The selection process gives a description of the blow-up solutions of (\ref{seq}) and in each of the bubbling disks, at least one component has energy with positive lower bound:
\begin{align*}
\frac{1}{2\pi}\int_{B_{l_j^k}(x_j^k)}h_1^ke^{u_k}>4 \quad \mathrm{or} \quad \frac{1}{2\pi}\int_{B_{l_j^k}(x_j^k)}h_2^ke^{-2u_k}>2.
\end{align*}
The fourth conclusion in Proposition \ref{pr2.1} provides us a control on the upper bound of the behavior of blow-up solutions outside the bubbling disks.

\begin{proposition}
\label{pr2.2}
For all $x_0\in B_1\setminus\Sigma_k,$ there exists $C_0$ independent of $x_0$ and $k$ such that
\begin{align*}
|u_k(x_1)-u_k(x_2)|\leq C_0\qquad\forall x_1,x_2\in B_{d(x_0,\Sigma_k)/2}(x_0).
\end{align*}
\end{proposition}

\medskip

The latter result is a Harnack-type estimates which describes the behavior of blowup solutions away from the blow-up points. In particular, let $x_k\in\Sigma_k$ and $\tau_k=\frac12d({x_k,\Sigma_k\setminus\{x_k\}})$, then for $x,y\in B_{\tau_k}(x_k)$ and $|x-x_k|=|y-x_k|$ we have $u_k(x)=u_k(y)+O(1)$ and hence $u_k(x)=\overline{u}_{x_k}(r)+O(1)$ where $r=|x_k-x|$ and
$$\overline{u}_{x_k}(r)=\frac{1}{2\pi}\int_{\partial B_r(x_k)}u_k.$$

We say $u_k$ or $-2u_k$ has fast decay on $x\in B_1$ if
$$u_k(x)+2\log\mathrm{dist}(x,\Sigma_k)\leq -N_k \quad \mbox{or} \quad -2u_k(x)+2\log\mathrm{dist}(x,\Sigma_k)\leq -N_k$$
hold for some $N_k\rightarrow\infty$, respectively. On the other hand, if
$$u_k(x)+2\log\mathrm{dist}(x,\Sigma_k)\geq -C \quad \mbox{or} \quad -2u_k(x)+2\log\mathrm{dist}(x,\Sigma_k)\geq-C,$$
for some $C>0$ independent of $k$, we say $u_k$ or $-2u_k$ has a slow decay at $x$, respectively. It is known from the following lemma that it is possible to choose $r$ such that both $u_k,-2u_k$ have fast decay property.
\begin{lem}
\label{le2.1}
For all $\e_k\rightarrow0$ with $\Sigma_k\subset B_{\e_k/2}(0),$ there exists $l_k\rightarrow0$ such that $l_k\geq 2\e_k$ and
\begin{equation*}
\max\{\bar u_k(l_k),-2\bar u_k(l_k)\}+2\log l_k\rightarrow-\infty,
\end{equation*}
where we are using the notation $\bar u_k(r):=\frac{1}{2\pi r}\int_{\p B_r} u_k.$
\end{lem}
Moreover, when we analyze the behavior of $u_k$ in each bubbling disk, we can still choose some ball in this bubbling disk such that both $u_k,-2u_k$ have fast decay property on the boundary of this ball, see Remark \ref{re2.1}.

\medskip

From \cite{jwy} and \cite{lin-wei-zh} we have seen that the definition of the fast and slow decay is essential for evaluating Pohozaev identities for equation \eqref{seq}. On the other hand, the Pohozaev identity also plays an important role for evaluating the local mass.  By straightforward computations we have the following Pohozaev identity:
\begin{align}
\label{2.2}
\begin{split}
&\int_{B_r}\left(x\cdot\nabla h_1^ke^{u_k}+\frac12 x\cdot\nabla h_2^ke^{-2u_k}\right)+\int_{B_r}\left(2h_1^ke^{u_k}+h_2^ke^{-2u_k}\right)\\
&=\int_{\partial B_r}r\left(|\partial_{\nu}u_k|^2-\frac12|\nabla u_k|^2\right)+\int_{\partial B_r}r\left(h_1^ke^{u_k}+\frac12h_2^ke^{-2u_k}\right).
\end{split}
\end{align}
It is possible to choose suitable $r=l_k\rightarrow 0$ such that
$$\frac{1}{2\pi}\int_{B_{l_k}}h_1^ke^{u_k}=\sigma_1+o(1), \qquad \frac{1}{2\pi}\int_{B_{l_k}}h_2^ke^{-2u_k}=\sigma_2+o(1)$$
and such that both $u_k,-2u_k$ have fast decay property on $\partial B_{l_k},$ where $\sigma_i$ are introduced in \eqref{localmass}. We point out that the fast decay property is important because it leads to the second term on the right hand side of \eqref{2.2} is $o(1)$. By making full use of \eqref{2.2} we can get
\begin{equation}
\label{2.3}
(\sigma_1-\sigma_2)^2=4\left(\sigma_1+\frac{\sigma_2}{2}\right).
\end{equation}
For the detail proof of (\ref{2.3}) we refer the readers to  Proposition 3.1 in \cite{jwy}. Furthermore, we have the following remark which will be used frequently in the forthcoming argument.

\begin{remark}
\label{re2.1}
We have already observed that the fast decay property is crucial in evaluating the Pohozaev identity \eqref{2.2}. Moreover, let $\Sigma_k'\subseteq\Sigma_k$ and assume that
$$\mathrm{dist}\bigr(\Sigma_k',\partial B_{l_k}(p_k)\bigr)=o(1)\,\mathrm{dist}\bigr(\Sigma_k\setminus\Sigma_k',\partial B_{l_k}(p_k)\bigr).$$
If both components $u_k, -2u_k$ have fast decay on $\partial B_{l_k}(p_k)$, namely
$$
	\max\{u_k(x),-2u_k(x)\} \leq -2\log|x-p_k| -N_k, \qquad  \mbox{for } x\in \partial B_{l_k}(p_k),
$$
for some $N_k\to+\infty$. Then, we can evaluate a local Pohozaev identity as in \eqref{2.2} and get
\begin{align*}
\bigr(\tilde{\sigma}_1^k(l_k)-\tilde{\sigma}_2^k(l_k)\bigr)^2=4\left(\tilde{\sigma}_1^k(l_k)+\frac{\tilde{\sigma}_2^k(l_k)}{2}\right)+o(1),
\end{align*}
where
\begin{align*}
&\tilde{\sigma}_1^k(l_k) = \frac{1}{2\pi}\int_{B_{l_k}(p_k)}h_1^k e^{u_k} \\
&\tilde{\sigma}_2^k(l_k) = \frac{1}{2\pi}\int_{B_{l_k}(p_k)}h_2^k e^{-u_k}.
\end{align*}
Observing that if $B_{l_k}(p_k) \cap \Sigma_k=\emptyset$ then $\tilde{\sigma}_i^k(l_k)=o(1), i=1,2$ and the above formula clearly holds.
\end{remark}

\medskip

Finally, we state a recent result on the total energy of the following equation
\begin{equation}
\label{2.4}
\left\{\begin{array}{l}
\Delta u+2 e^u=\sum_{j=1}^N4\pi n_j\delta_{p_j} \qquad \mathrm{in}~\mathbb{R}^2,\\
\int_{\mathbb{R}^2}e^u<+\infty,
\end{array}\right.
\end{equation}
given in \cite{lin-wei-zh2}, which will be used later on.
\medskip

\noindent {\bf Theorem A.} \cite{lin-wei-zh2}
Let $u$ be a solution of equation (\ref{2.4}), where $p_j,\,j=1,\cdots,N$ are distinct points in $\mathbb{R}^2$ and $n_j\in\mathbb{N},j=1,\cdots,N.$ Then
\begin{align*}
\frac{1}{2\pi}\int_{\mathbb{R}^2}e^u=2n,
\end{align*}
for some $n\in\mathbb{N}.$

\medskip

\subsection{Description of solutions around each blow-up point.}
To understand the concentration of energy we start from any fixed point of $\Sigma_k.$ Without loss of generality we may assume that $0\in\Sigma_k$ by considering a suitable translation. Let $\tau_k=\frac12\mathrm{dist}(0,\Sigma_k\setminus\{0\})$. Let
$$\sigma_1^k(r)=\frac{1}{2\pi}\int_{B_r(0)}h_1^ke^{u_k}, \qquad \sigma_2^k(r)=\frac{1}{2\pi}\int_{B_r(0)}h_2^ke^{-2u_k},$$
for $0<r\leq\tau_k$ and $\overline{u}_k(r)=\frac{1}{2\pi r}\int_{\partial B_r(0)}u_k.$ By using equation \eqref{seq} we get the following key property:
\begin{align} \label{deriv}
\frac{d}{dr}\overline{u}_k(r)=\frac{1}{2\pi r} \int_{\p B_r} \frac{\p u_k}{\p \nu} = \frac{1}{2\pi r} \int_{B_r} \D u_k = \frac{-\sigma_1^k(r)+\sigma_2^k(r)}{r}.
\end{align}
Moreover, from the selection process we have
\begin{align*}
\max\{u_k(x),-2u_k(x)\}+2\log|x|\leq C, \qquad |x|\leq \tau_k.
\end{align*}
If both components have fast decay on $\partial B_r(0)$ for $r\in (0,\tau_k)$, then $\sigma_1^k(r),\sigma_2^k(r)$ satisfy
\begin{align*}
(\sigma_1^k(r)-\sigma_2^k(r))^2=4\left(\sigma_1^k(r)+\frac{\sigma^k_2(r)}{2}\right),
\end{align*}
see Remark \ref{re2.1}. Furthermore, we have the following result on the description of the energy concentration in $B_{\tau_k}(0).$

\begin{proposition}
\label{pr2.3}
Suppose (\ref{seq})-(\ref{a3}) hold for $u_k$ and $h_i^k$. For any $r_k$ in $(0,\tau_k)$ such that both $u_k,-2u_k$ have fast decay on $\partial B_{r_k},$ i.e,
$$\max\{u_k(x),-2u_k(x)\}\leq-2\log|x|-N_k, \qquad \mathrm{for}~|x|=r_k~\mathrm{and~some}~N_k\rightarrow\infty,$$
we have that $(\sigma_1^k(r_k),\sigma_2^k(r_k))$ is a small perturbation of one of the two following types:
\begin{align*}
\bigr(2m(3m-1),2(3m-1)(m-1)\bigr) \quad \mathrm{or} \quad \bigr(2(3m-2)(m-1),2(3m-5)(m-1)\bigr),
\end{align*}
for some $m\in \mathbb{Z}$, excluding the case $(0,0)$. In particular, the first component is multiple of $4+o(1)$ while the second component is multiple of $2+o(1)$.

On $\partial B_{\tau_k}$ either both $u_k, -2u_k$ have fast decay, or one component has fast decay while the other one has not fast decay. Suppose for example $-2u_k$ has not the fast decay property, i.e.
\begin{align*}
-2u_k(x)+2\log|x|\geq-C, \qquad \mbox{for } |x|=\tau_k \mbox{ and some } C>0,
\end{align*}
while for $u_k$ it holds
\begin{align*}
u_k(x) \leq-2\log|x|-N_k, \qquad \mbox{for } |x|=s_k \mbox{ and some } N_k\rightarrow\infty.
\end{align*}
Then we have $\sigma_1^k(\tau_k)\in 4\mathbb{N} + o(1)$ (in the other case we have $\sigma_2^k(\tau_k)\in 2\mathbb{N} + o(1)$).

In particular, in any case the local energy in $B_{\tau_k}$ of at least one of the two components $u_k, -2u_k$ is a perturbation of a multiple of $4$ (for the first component) or $2$ (for the second component).
\end{proposition}

\begin{proof}
The proof is mainly followed by the argument in \cite[Proposition 5.1]{lin-wei-zh} and \cite[Proposition 4.1]{jwy}, we will only give the key steps here. Let $-2\log\delta_k=\max_{x\in B_{\tau_k}(0)}\max\{u_k(x),-2u_k(x)\}.$ Let us define
\begin{align*}
\begin{split}
&v_1^k(y)=u_k(\delta_ky)+2\log \delta_k, \\
&v_2^k(y)=-2u_k(\delta_ky)+2\log\delta_k,
\end{split} \qquad |y|\leq\tau_k/\delta_k.
\end{align*}
As in Proposition \ref{pr2.1} it holds that one of $v_i^k$ converges and the other one tends to minus infinity over the compact subsets of $\mathbb{R}^2$. Without loss of generality we assume that $v_1^k$ converges to $v_1$ in $C_{loc}^2(\mathbb{R}^2)$ and $v_2^k$ converges to minus infinity over any compact subset of $\R^2.$ Then by the quantization of the limit function $\D v_1 + e^{v_1}=0$ in $\R^2$ we can choose $R_k\rightarrow\infty$ such that
\begin{equation}
\label{2.5}
\frac{1}{2\pi}\int_{B_{R_k}}h_1^k(\delta_ky)\,e^{v_1^k}=4+o(1), \qquad \frac{1}{2\pi}\int_{B_{R_k}}h_2^k(\delta_ky)\,e^{v_2^k}=o(1).
\end{equation}
For $r\geq R_k$ we clearly have
\begin{equation*}
\sigma_i^k(\delta_kr)=\frac{1}{2\pi}\int_{B_r}h_i^k(\delta_ky)\,e^{v_i^k}.
\end{equation*}
Then we get $\sigma_1^k(\delta_kR_k)=4+o(1)$ and $\sigma_2^k(\delta_kR_k)=o(1).$ Now, we consider the energy change from $B_{\delta_kR_k}$ to $B_{\tau_k}$. First we note that on $\partial B_{\delta_kR_k}$, by \eqref{deriv} and the latter estimate of the local energies we get
$$\frac{d}{dr}(-2\overline{u}_k(r)+2\log r)>0,$$
which means that $-2u_k$ may become a slow decay component when $r$ increases. The first possibility in $B_{\tau_k}$ is that
\begin{equation}
\label{2.6}
\sigma_1^k(\tau_k)=4+o(1), \qquad \sigma_2^k(\tau_k)=o(1),
\end{equation}
which means $-2u_k$ does not become a slow decay component. Indeed, \cite[Lemma 4.1]{jwy} it is proved for the sinh-Gordon equation that if no component changes to be a slow decay component, the energy of each component only changes by $o(1).$ We can modify the argument to get the same conclusion for the Tzitz$\acute{e}$ica equation.

\medskip

Suppose now $-2u_k$ become a slow decay component before $r$ reaches $\tau_k$. Suppose at some $s>r,$
\begin{equation*}
-2\bar u_k(s)+2\log s\geq -C
\end{equation*}
for some $C>0$ very large. Observe that $-2u_k$ starts to increase its energy but the energy of $u_k$ can not change much because $\frac{d}{dr}(\bar u_k(r)+2\log r)$ is still negative. If $\tau_k/s\rightarrow\infty,$ which means that $\tau_k$ is very large compared with $s$, then we can find $N>1$ such that on $\partial B_{Ns}$
\begin{align}
&\sigma_2^k(Ns)\geq 5, \qquad \sigma_1^k(Ns)=4+o(1),\label{est}\\
&\bar u_k(Ns)+2\log(Ns)\leq-N_k, \qquad \mathrm{for~some}~N_k\rightarrow\infty,\nonumber\\
&\frac{d}{dr}(-2\bar u_k(r)+2\log r)\mid_{r=Ns}<0, \qquad \frac{d}{dr}(\bar u_k(r)+2\log r)\mid_{r=Ns}>0,\nonumber
\end{align}
see \cite[Lemma 4.2]{jwy}. Roughly speaking, from $r=s$ to $r=Ns$ the energy of $-2u_k$ increase and hence the derivative of $-\bar  u_k(r)+2\log(r)$ changes from positive to negative. By Proposition \ref{pr2.2} it is possible to show that $u_k$ it still has fast decay and hence its energy does not change. Since $\tau_k/s\rightarrow\infty$ we can find $N_k'$ tending to $+\infty$ slowly such that $N_k's\leq\tau_k/2$ and on $\p B_{N_k's}$ both $u_k$ and $-2u_k$ have fast decay property, see again \cite[Lemma 4.2]{jwy}. Evaluating the Pohozaev identity on $\partial B_{N_k's}$, see Remark \ref{re2.1}, and taking into account the estimate \eqref{est} we deduce
\begin{equation} \label{est2}
\sigma_{1}^k(N_k's)=4+o(1), \qquad \sigma_2^k(N_k's)=10+o(1).
\end{equation}
For the case $\tau_k$ is only comparable to $s$, then on $\p B_{\tau_k}$ we have $-2u_k$ is a slow decay component and $\sigma_1^k(\tau_k)=4+o(1).$

\medskip

On $\p B_{N_k's}$ we have by \eqref{deriv} and \eqref{est2}
\begin{align*}
\frac{d}{dr}(\bar u_k(r)+2\log r)=\frac{6+o(1)}{r}, \qquad \frac{d}{dr}(-2u_k(r)+2\log r)=\frac{-12+o(1)}{r}, \qquad r=N_k's.
\end{align*}
At this point the role of $u_k, -2u_k$ is exchanged and $u_k$ may become a slow decay component for large $r$. Then as $r$ grows from $N_k's$ to $\tau_k$, as before either there is no change of energy up to $\tau_k$ or $u_k$ change to be slow decay at some $N_k's \ll \tau_k$ while $-2u_k$ does not changed its energy: in the latter case we can find as before $L_k$ with $L_kN_k's\leq\frac{\tau_k}{2}$ such that both components have fast decay on $\p B_{L_kN_k's}$. The Pohozaev identity gives us
\begin{equation*}
\sigma_1^k(L_kN_k's)=20+o(1), \qquad \sigma_2^k(L_kN_k's)=10+o(1).
\end{equation*}
Since after each step one of the local masses changes by a positive number, using the uniform bound on the energy the process stops after finite steps. Eventually we can get Proposition \ref{pr2.3}.
\end{proof}

\medskip

\subsection{Bubbling groups and the conclusion}
After analyzing the behavior of the bubbling solution $u_k$ in each bubbling disk, we turn to consider the combination of the bubbling disks in a group. The concept of group for this kind of problems was first introduced in \cite{lin-wei-zh}. Roughly speaking, the groups are made of points in $\Sigma_k$ which are relatively close to each other but relatively far away from the other points in $\Sigma_k$.

\medskip

\noindent {\bf Definition.} Let $G=\{p_1^k,\cdots,p_q^k\}$ be a subset of $\Sigma_k$ with more than one point in it. $G$ is called a group if
\begin{enumerate}
  \item dist$(p_i^k,p_j^k)\sim $ dist$(p_s^k,p_t^k)$,
 where $p_i^k,p_j^k,p_s^k,p_t^k$ are any points in $G$ such that $p_i^k\neq p_j^k$ and $p_t^k\neq p_s^k.$ \vspace{0.3cm}

  \item $\dfrac{\mbox{dist}(p_i^k,p_j^k)}{\mbox{dist}(p_i^k,p_k)}\rightarrow0$,
  for any $p_k\in\Sigma_k\setminus G$ and for all $p_i^k,p_j^k\in G$ with $p_i^k\neq p_j^k.$
\end{enumerate}

\medskip

We start by noticing that by Proposition \ref{pr2.2} if both $u_k, -2u_k$ have fast decay around one of the disks in a group, they are forced to have fast decay around all the disks of this group. More precisely, take $B_{\tau_1^k}(x_1^k),\cdots,B_{\tau_m^k}(x_m^k)$ in some group, where $\tau_j^k=\frac12\mathrm{dist}(x_j^k,\Sigma_k\setminus\{x_j^k\})$. By the definition of group, all the $\tau_l^k,~l=1,\cdots,m$ are comparable. Suppose both $u_k, -2u_k$ have fast decay: then we can find $N_k\rightarrow\infty$ such that all the disks in this group are contained for example in $B_{N_k\tau_1^k}(0)$ and
\begin{align*}
\sigma_1^k\left(N_k\tau_1^k\right)=\sum_{j=1}^m\sigma_1^k\left(B_{\tau_{j}^k}(x_j^k)\right)+o(1),\qquad \sigma_2^k\left(N_k\tau_1^k\right)=\sum_{j=1}^m\sigma_2^k\left(B_{\tau_{j}^k}(x_j^k)\right)+o(1),
\end{align*}
where
$$\sigma_1^k\left(B_{\tau_{j}^k}(x_j^k)\right)=\frac{1}{2\pi}\int_{B_{\tau_{j}^k}(x_j^k)}h_1^ke^{u_k}, \qquad
\sigma_2^k\left(B_{\tau_{j}^k}(x_j^k)\right)=\frac{1}{2\pi}\int_{B_{\tau_{j}^k}(x_j^k)}h_2^ke^{-2u_k}.$$
Roughly speaking, the energy contribution comes just from the energy in the bubbling disks. Since both components have fast decay, Proposition \ref{pr2.3} asserts that around each bubbling disk the local energy of $u_k$ is multiple of $4+o(1)$ and the one of $-2u_k$ is multiple of $2+o(1)$. From the above formula we get that $\sigma_1^k(N_k\tau_{1,k})$ is multiple of $4+o(1)$ and $\sigma_2^k(N_k\tau_{1,k})$ is multiple of $2+o(1)$. By using the Pohozaev identity again, we can get $\left(\sigma_1^k(N_k\tau_{1,k}),\sigma_2^k(N_k\tau_{1,k})\right)$ is a small perturbation of
$$\bigr(2n_k(3n_k-1),2(3n_k-1)(n_k-1)\bigr) \quad \mathrm{or} \quad \bigr(2(3n_k-2)(n_k-1),2(3n_k-5)(n_k-1)\bigr),$$
for some $n_k\in\mathbb{Z}.$

\medskip

We consider now the more difficult case when only one component has fast decay in a group. Suppose $u_k$ has fast decay, while $-2u_k$ has slow decay in each bubbling disk (the other case can be studied similarly). In this case Proposition \ref{pr2.3} yields $\sum_{j=1}^m\sigma_1^k(B_{\tau_j^k}(x_j^k))$ is a multiple of $4+o(1)$. Using Proposition~\ref{pr2.2} again we can choose $N_k\rightarrow\infty$ so that $u_k$ is still fast decay on $\partial B_{N_k\tau_k}(0)$ and
\begin{align*}
\sigma_1^k\bigr(B_{N\tau_k}(0)\bigr)=\sum_{j=1}^m\sigma_1^k\left(B_{\tau_j^k}(x_j^k)\right)+o(1).
\end{align*}
On the other hand, we consider the following scaling
\begin{align*}
\begin{split}
&\tilde{v}_1^k(y)=u_k(\tau_ky)+2\log\tau_k, \\
&\tilde{v}_2^k(y)=-2u_k(\tau_ky)+2\log\tau_k,
\end{split} \qquad |y|\leq N_k\tau_k^{-1}.
\end{align*}
Then we have
\begin{align}
\label{2.7}
\Delta_y\tilde{v}_1^k(y)+\tilde{h}_1^k(y)\,e^{\tilde{v}_1^k(y)}-\tilde{h}_2^k(y)\,e^{\tilde{v}_2^k(y)}=0, \qquad |y|\leq N_k\tau_k^{-1},
\end{align}
where $\tilde{h}_i^k(y)=h_i^k(\tau_ky),~i=1,2.$ Let $q_j^k$ be the images of $x_j^k$ after scaling. One can see that $\tilde{v}_1^k$ has still fast decay on $\partial B_{N_k}(0)$ while $\tilde{v}_2^k$ has slow decay on it. Moreover, it is not difficult to show that $\tilde{h}_1^ke^{\tilde{v}_1^k}\rightharpoonup \sum_{j=0}^m8\pi\, n_{1,j}\delta_{q_j}$, where $n_{1,j}\in \mathbb{N}$, while $\tilde{h}_2^ke^{\tilde{v}_2^k}\rightharpoonup\sum_{j=0}^m4\pi\, n_{2,j}\delta_{q_j}+F$, where $n_{2,j}\in \mathbb{N}$ and $F\in L^1(\R^2)$. Then, we want to prove that the integral of $F$ in $\R^2$ is multiple of $4\pi.$

\medskip

Since $\tilde{v}_1^k$ has fast decay and $\tilde{v}_2^k$ has slow decay, one can see from the argument of Proposition \ref{pr2.3} that it holds $8\pi \,n_{1,j}>4\pi \, n_{2,j}$. Recalling equation \eqref{2.7}, by the above argument we can see that $\tilde{v}_2^k$ converges to $\tilde{v}_2$ in $\R^2\setminus\{q_1,\dots,q_m\}$, where $\tilde{v}_2$ satisfies
\begin{equation}
\label{2.8}
\Delta\tilde{v}_2+2e^{\tilde{v}_2}=\sum_{j=0}^m4\pi \,\tilde{n}_j\delta_{q_j} \qquad \mathrm{in}~\R^2,
\end{equation}
where $\tilde{n}_j\in\mathbb{N}.$ Exploiting the global quantization of Theorem A we get $\frac{1}{2\pi}\int_{\R^2}e^{\tilde{v}_2}=2\tilde{n}$ for some $\tilde{n}\in\mathbb{N}.$ This gives the quantization of $F$ and therefore we get $\sigma_2^k(\tau_kL)=2n+o(1)$ for some $n\in\mathbb{N}.$

\medskip

Let $2\tau_kL_k$ be the distance from the group we are considering to the nearest group different from it. By the definition of group we have $L_k\rightarrow\infty.$ As before we can find $\tilde{L}_k\leq L_k,~\tilde{L}_k\rightarrow\infty$ slowly such that the energy of $\tilde{v}_i^k$ in $B_{\tilde{L}_k}(0)$ does not change so much and both $\tilde{v}_i^k$ have fast decay on $\partial B_{\tilde{L}_k}(0)$:
\begin{align} \label{es3}
&\sigma_1^k(\tau_k \tilde L_k)=4\bar n + o(1), \qquad \sigma_2^k(\tau_k \tilde L_k)=2 n + o(1), \qquad \mbox{for some } \bar n, n \in\mathbb{N}
\end{align}
and
$$
\tilde{v}_i^k(y) \leq -2\log\tilde{L}_k-N_k, \qquad \mbox{for } |y|=\tilde{L}_k, \quad i=1,2,
$$
for some $N_k\to+\infty$. Since on $\partial B_{\tilde{L}_k}$ both components $\tilde{v}_1^k, \tilde{v}_2^k$ have fast decay, we can compute the local Pohozaev identity of Remark \ref{re2.1}2. Using the estimate \eqref{es3} we get that $(\sigma_1^k(\tau_k\tilde{L}_k),\sigma_2^k(\tau_k\tilde{L}_k))$ is a $o(1)$ perturbation of one of the two following types:
\begin{equation}\label{2.9}
\bigr(2\tilde m(3\tilde m-1),2(3\tilde m-1)(\tilde m-1)\bigr) \quad \mbox{or} \quad \bigr(2(3\tilde m-2)(\tilde m-1),2(3\tilde m-5)(\tilde m-1)\bigr),
\end{equation}
for some $\tilde m\in\mathbb{Z}.$ Now, as $r$ grows from $\tau_k\tilde{L}_k$ to $\tau_kL_k$ (and we reach the second group) we can follow the same argument of Proposition \ref{pr2.3} to get analogous conclusions. In particular, in any case at the boundary of a group at least one of the two components $u_k, -2u_k$ has fast decay and the local energy in this group of such component is a perturbation of a multiple of $4$ (for the first component) or $2$ (for the second component).

\medskip

\no \emph{Proof of Theorem \ref{th:quantization}.}
To get the desired energy quantization we are left with combining the groups. The process is very similar to the combination of bubbling disks as we have done before and we refer the readers to \cite{jwy} for more details. With the combination of the groups we further include the groups which are far away from $0$. From the selection process it is known that we have only finite bubbling disks and as a result the combination procedure will terminate in finite steps. Finally, we can take $s_k\rightarrow0$ with $\Sigma_k\subset B_{s_k}(0)$ such that both components $u_k,-2u_k$ have fast decay on $\p B_{s_k}(0).$ Therefore, we have that $\sigma_1^k(s_k),\sigma_2^k(s_k)$ is a small perturbation of one of the two types:
$$\bigr(2m(3m-1),2(3m-1)(m-1)\bigr), \qquad \bigr(2(3m-2)(m-1),2(3m-5)(m-1)\bigr),$$
for some $m\in\mathbb{Z}.$ On the other hand, we have
$$\sigma_i=\lim_{k\rightarrow\infty}\sigma_i^k(s_k),\qquad i=1,2.$$
It follows that $\sigma_1,\sigma_2$ satisfy the quantization property of Theorem \ref{th:quantization} and we finish the proof.

\

\section{Exclusion of boundary blow-up} \label{sec:boundary}
In this section we consider the Dirichlet problem in a bounded smooth domain $\Omega \subset \R^2$, see \eqref{eq2}, and we shall prove that the blow-up phenomenon can not occur on the boundary $\partial \Omega$, see Theorem \ref{th:boundary}. Let
\begin{equation}
\label{3.1}
M_k(x)=\max \{u_k(x),-2u_k(x)\},
\end{equation}
and $p_k\in\overline{\Omega}$ be such that $M_k(p_k)=\max_{\overline{\Omega}}M_k(x).$ We set then $\mu_k$ to be such that
$$-2\ln \mu_k=M_k(p_k).$$
We first note that $\mu_k\rightarrow 0.$ If not, by Green representation for \eqref{eq2} we have $|u_k|$ is uniformly bounded in $\o,$ which contradicts the assumption $|u_k|\to+\infty.$ On the other hand, by the boundary condition, we see that $p_k\notin \p\o$. Indeed, we can further show that $p_k$ must have some distance from the boundary.

\begin{lem}
\label{le3.1}
It holds that
$$\mathrm{dist}(p_k,\p\o)/\mu_k\rightarrow+\infty.$$
\end{lem}

\begin{proof}
We prove it by contradiction. Suppose we can find a sequence $(p_k,\mu_k)$ such that $\mathrm{dist}(p_k,\p\o)=O(\mu_k)$. Consider the dilated set
$$\o_k=(\o-p_k)/\mu_k.$$
Without loss of generality we may suppose that $\o_k\rightarrow(-\infty,t_0)\times\R$ and $u_k(p_k)=-2\ln\mu_k$. Let
\begin{align*}
&v_k(y)=u_k(p_k+\mu_k y)+2\ln\mu_k+\ln h_1^k(p_k), \\
&w_k(y)=-2u_k(p_k+\mu_k y)+2\ln\mu_k+\ln h_2^k(p_k).
\end{align*}
Let $R>0$ and $y\in B_R(0)\cap\o_k.$ Using the representation formula in \eqref{eq2} we obtain
\begin{align}
\label{3.2}
|\nabla v_k|&=|\mu_k\nabla u_k(p_k+\mu_ky)|\nonumber\\
&=\mu_k\left|\int_{\o}\nabla G(p_k+\mu_ky,z)\left(h_1^ke^{u_k}(z)-h_2^ke^{-2u_k}(z)\right)\mathrm{d}z\right|\\
&\leq C\mu_k\left(\int_{B_{2R\mu_k}(p_k)}+\int_{\Omega\setminus B_{2R\mu_k}(p_k)}\right)
\frac{\left|h_1^ke^{u_k}(z)-h_2^ke^{-2u_k}(z)\right|}{\left|p_k+\mu_ky-z\right|}\,\mathrm{d}z.\nonumber\\
\nonumber\end{align}
Now, in $B_{2R\mu_k}$ we simply have $\max\{e^{u_k},e^{-2u_k}\}\leq e^{u_k(p_k)}=\mu_k^{-2}$, while in $\o_k\setminus B_{2R}$ we have the estimate
$$|p_k+\mu_ky-z|\geq|z-p_k|-\mu_k|y|\geq R\mu_k.$$
Hence,
\begin{align*}
|\nabla{v}_k|\leq C\mu_k\int_{B_{2R\mu_k}(p_k)}\frac{|h_1^ke^{u_k}-h_2^ke^{-2u_k}|}{|p_k+\mu_ky-z|}
+C_R\int_{\Omega}\left|h_1e^{u_k}-h_2e^{-2u_k}\right|\leq C_R.
\end{align*}
It follows that $|\nabla v_k|\leq C_R$ in $B_R(0)\cap\o_k$. Therefore, we deduce that
$$|v_k(y)-v_k(0)|\leq C|y|\leq C \qquad \forall y\in\overline{B_R(0)\cap\o_k}.$$
But taking a point $y_0\in \p\o_k$ we get
$$|u_k(p_k)|=|v_k(y_0)-v_k(0)|\leq C.$$
By construction we conclude
$$-\ln\mu_k=O(1),$$
which contradicts to the fact $\mu_k\rightarrow 0$. The proof of the lemma is done.
\end{proof}

\medskip

Recalling the definitions of $v_k, w_k$ in Lemma \ref{le3.1} we observe $v_k+w_k\leq 4\ln\mu_k+C$. Moreover, it is easy to see that one of $v_k,w_k$ is locally uniformly bounded on compact subsets of $\R^2$: without loss of generality we can assume $v_k$ is locally uniformly bounded. It follows that $w_k$ tends to $-\infty$ uniformly on compact subset of $\R^2.$ Therefore,  $v_k$ converges to a function $v$ which satisfies the Liouville equation
$$
\Delta v+e^v=0 \qquad \mbox{in } \R^2
$$
We point out that by the quantization result of the latter equation we can deduce the following bound on the local energy:
\begin{equation} \label{int1}
\lim_{r\rightarrow0}\lim_{k\rightarrow\infty}\int_{B_r(p_k)}h_1^ke^{v_k}\geq 8\pi.
\end{equation}
Notice moreover that in case $w_k$ converges to a function we have instead
\begin{equation} \label{int2}
\lim_{r\rightarrow0}\lim_{k\rightarrow\infty}\int_{B_r(p_k)}h_2^ke^{w_k}\geq 4\pi.
\end{equation}
In order to localize the previous arguments we give now a definition similar to the selection process in Proposition \ref{pr2.1}, which plays an important role in the following arguments. This kind of approach can be found also in \cite{ro-wei,lwy,wei-ya}.

\medskip

\noindent {\bf Definition.}
We say that the property $\mathcal{H}_m$ holds if there exist points $\{p_{k,1},\cdots,p_{k,m}\}$ such that, letting
$$\mu_{k,j}=e^{-\frac12\max\{u_k(p_{k,j}),-2u_k(p_{k,j})\}}\rightarrow0, \qquad \mbox{as } k\to+\infty, \; \forall j=1,\dots,m,$$
we have
\begin{enumerate}
  \item $\lim_{k\rightarrow\infty}\frac{|p_{k,i}-p_{k,j}|}{\mu_{k,i}}=+\infty$ for any $i\neq j$,
  \item $\lim_{k\rightarrow\infty}\mathrm{dist}(p_{k,i},\p\o)/\mu_{k,i}=+\infty$ for all $i=1,\cdots,m,$
  \item for all $i=1,\cdots,m$, letting
  \begin{align*}
  &v_{k,i}(y)=u_k(p_{k,i}+\mu_{k,i}y)+2\ln\mu_{k,i}+\ln h_1^k(p_{k,i}),\\
  &w_{k,i}(y)=-2u_k(p_{k,i}+\mu_{k,i}y)+2\ln\mu_{k,i}+\ln h_2^k(p_{k,i}),
  \end{align*}
   then, in any compact subset of $\R^2$ either $v_{k,i}$ converges to a solution of $\Delta v+e^v=0$ while $w_{k,i}$ tends to $-\infty$ on compact subsets of $\R^2$ or $w_{k,i}$ converges to a solution of $\Delta w+2e^w=0$ while $v_{k,i}$ tends to $-\infty$ on compact subsets of $\R^2.$
\end{enumerate}

\medskip

By the above arguments we have that $\mathcal{H}_1$ holds. From \eqref{int1} and \eqref{int2} we observe that if $\mathcal{H}_m$ holds, then for every $i=1,2,\cdots,m,$
\begin{align}
\label{3.3}
\lim_{r\to 0}\lim_{k\rightarrow\infty}\int_{B_r(p_{k,i})}h_1^ke^{u_k}+\lim_{r\to 0}\lim_{k\rightarrow\infty}\int_{B_r(p_{k,i})}h_2^ke^{-2u_k}\geq 4\pi.
\end{align}

We start by showing the following fact.
\begin{lem} \label{l:iter}
Suppose $\mathcal{H}_l$ holds. Then we have the following alternative: either $\mathcal{H}_{l+1}$ holds or there exists $C>0$ such that
\begin{equation}
\label{3.4}
\inf_{i=1,\cdots,l}|x-p_{k,i}|^2e^{\max\{u_k(x),-2u_k(x)\}}\leq C.
\end{equation}
\end{lem}

\begin{proof}
Recall the definition of $M_k(x)$ in (\ref{3.1}). Let $\Gamma_k(x)=\inf_{i=1,\cdots,l}|x-p_{k,i}|^2e^{M_k(x)}$ we suppose that $\|\Gamma_k\|_{L^{\infty}(\o)}\rightarrow+\infty$ and we have to prove that $\mathcal{H}_{l+1}$ holds true. Observing that $M_k(x)\mid_{\p\o}=0$ we may write $\Gamma_k(x_k)=\max_{\O}\Gamma_k(x)$. Setting $\gamma_k=e^{-M_k(x_k)/2}$ we have $\gamma_k\rightarrow0$ and
$\Gamma_k(x_k)=\inf_{i=1,\cdots,l}|x_k-p_{k,i}|^2/\gamma_k^2\rightarrow+\infty.$ As a consequence we get
\begin{equation}
\label{3.5}
\frac{|x_k-p_{k,i}|}{\gamma_k}\rightarrow+\infty, \qquad \mathrm{for~all}~i=1,\cdots,l.
\end{equation}

We claim
\begin{equation}
\label{3.6}
\frac{|x_k-p_{k,i}|}{\mu_{k,i}}\rightarrow+\infty, \qquad \mathrm{for~all}~i=1,\cdots,l.
\end{equation}
If not, we can find some $j$ such that $x_k-p_{k,j}=O(\mu_{k,j}).$ We set $x_k=p_{k,j}+\mu_{k,j}\theta_{k,j}$ with $\theta_{k,j}=O(1)$. Then,
\begin{align*}
|x_k-p_{k,j}|^2e^{M_k(x_k)}=|\theta_{k,j}|^2e^{M_k(p_{k,j}+\mu_{k,j}\theta_{k,j})+2\ln\mu_{k,j}}\rightarrow C_{\theta_j}<+\infty
\end{align*}
and hence $\Gamma_k(x_k)=O(1)$ which is impossible. Therefore, the claim (\ref{3.6}) holds.

\medskip

Let $\epsilon\in(0,1)$ and consider the dilated set $\tilde{\o}_k=(\o-x_k)/\gamma_k$ Then, for any $y\in B_{R}(0)\cap\tilde{\o}_k$, we have
$$\Gamma_k(x_k+\gamma_ky)\leq \Gamma_k(x_k),$$
which implies
$$\inf_{i=1,\cdots,l}|x_k+\gamma_ky-p_{k,i}|e^{M_k(x_k+\gamma_ky)}\leq\inf_{i=1,\cdots,l}|x_k-p_{k,i}|e^{M_k(x_k)}.$$
Let
\begin{align*}
&v_{k1}(y)=u_k(x_k+\gamma_ky)+2\ln\gamma_k+\ln h_1^k(x_k), \\
&w_{k1}(y)=-2u_k(x_k+\gamma_ky)+2\ln\gamma_k+\ln h_2^k(x_k).
\end{align*}
Then we have
\begin{align*}
e^{\max\{v_{k1}(y),w_{k1}(y)\}}\leq C_1\frac{\inf_{i=1,\cdots,l}|x_k-p_{k,i}|^2}{\inf_{i=1,\cdots,l}|x_k+\gamma_ky-p_{k,i}|^2}\,,
\end{align*}
where $C_1>0$ depends just on $h^k_1,h^k_2$. By (\ref{3.5}) we are able to choose $k(R)\leq k$ such that $|x-p_{k,i}|/\gamma_k\geq\frac{R}{\epsilon}$ for all $i=1,\cdots,l$. By the triangle inequality it is easy to see that for all $i$ we have
$$|x_k+\gamma_ky-p_{k,i}|\geq(1-\epsilon)|x_k-p_{k,i}|,$$
which gives
\begin{align*}
\max\{v_{k1}(y),w_{k1}(y)\}\leq\ln\frac{C_1}{(1-\epsilon)^2}, \qquad \forall y\in B_R(0)\cap\o_k,~k\geq k(R)\,,
\end{align*}
and
\begin{align*}
e^{M_k(x_k+\gamma_ky)}\leq\frac{C_1}{(1-\epsilon)^2C_2}\gamma_k^{-2}\,, \qquad \forall y\in B_R(0)\cap\o_k,~k\geq k(R),
\end{align*}
where $C_2>0$ is taken so that $C_2\leq\min_{x\in\ov\O}\{h_1^k(x),h_2^k(x)\}$. By the same argument of Lemma \ref{le3.1} one can show
$$\frac{\mathrm{dist}(x_k,\p\o)}{\gamma_k}\rightarrow+\infty.$$
Based on the above facts it is sufficient to take $p_{k,l+1}=x_k$ and $\mu_{k,l+1}=\gamma_k$ to get the validity of the property $\mathcal{H}_{l+1}$, see its definition above.
\end{proof}

\medskip

Furthermore, it is possible to show that the above process finishes after a finite number of steps.
\begin{lem}
\label{le3.2}
There exists $m$ such that $\mathcal{H}_m$ holds and there exists $C>0$ such that
\begin{align*}
\inf_{i=1,\cdots,m}|x-p_{k,i}|^2e^{\max\{u_k(x),-2u_k(x)\}}\leq C, \qquad \forall x\in\o.
\end{align*}
\end{lem}

\begin{proof}
Suppose by contradiction the statement of the lemma is false. Since we know $\mathcal{H}_1$ holds, then by Lemma \ref{l:iter} $\mathcal{H}_l$ holds for all $l\geq1$. But for fixed $R>0$ we have
$$B_{R\mu_{k,i}}(p_{k,i})\cap B_{R\mu_{k,j}}(p_{k,j})=\emptyset$$
for $i\neq j$ and $k$ sufficiently large. It follows that
\begin{align*}
\int_{\o}h_1^ke^{u_k}+\int_{\o}h_2^ke^{-2u_k}\geq\sum_{i=1}^l\int_{B_{R\mu_{k,i}}(p_{k,i})}\left(h_1^ke^{u_k}+h_2^ke^{-2u_k}\right)\geq 4\pi l+o(1),
\end{align*}
for any $l\in\N$, where we used \eqref{3.3}, namely
\begin{equation}
\label{3.7}
\int_{B_{R\mu_{k,i}}(p_{k,i})}\left(h_1^ke^{u_k}+h_2^ke^{-2u_k}\right)\geq 4\pi+o(1).
\end{equation}
This fact is in contradiction with the energy bound (\ref{a3}). Therefore, the lemma holds.
\end{proof}

\medskip

A byproduct of the above selection process is the following estimate.
\begin{lem}
\label{le3.3}
Let $p_{k,i}$, $i=1,\dots, m$ be points as in Lemma \ref{le3.2}. Then, there exists $C>0$ such that
$$\inf_{i=1,\cdots,m}|x-p_{k,i}||\nabla u_k(x)|\leq C, \qquad \forall x\in\o.$$
\end{lem}

\begin{proof}
Using the representation formula for equation (\ref{eq2}) we get
\begin{align*}
|\nabla u_k|\leq C\int_\o\frac{1}{|x-z|}\left(h_1^ke^{u_k}(z)-h_2^ke^{-2u_k}(z)\right)\,\mathrm{d}z,
\end{align*}
where we used $|\nabla G(x,y)|\leq\frac{C}{|x-y|}.$ We decompose $\o$ into $\o=\cup_{i=1}^m\o_{k,i}$, where
$$\o_{k,i}=\bigr\{x\in\o:|x-p_{k,i}|=R_k(x)\bigr\},~i=1,\cdots,m, \qquad R_k(x):=\inf_{i=1,\cdots,m}|x-p_{k,i}|.$$
For $z\in\o_{k,i}\setminus B_{\frac{|x-p_{k,i}|}{2}}(p_{k,i})$ we have the estimate
\begin{align*}
|x-z|^{-1}e^{u_k(z)}\leq\frac{C}{|x-z||z-p_{k,i}|^2}\leq\frac{C}{|x-z||x-p_{k,i}|^2}\,.
\end{align*}
Hence we can conclude that
\begin{align}
\label{3.8}
\int_{\o_{k,i}\setminus B_{\frac{|x-p_{k,i}|}{2}}(p_{k,i})}\frac{h_1^ke^{u_k}(z)}{|x-z|}\,\mathrm{d}z\leq\frac{C}{|x-p_{k,i}|}\,.
\end{align}
While for $z\in\o_{k,i}\cap B_{\frac{|x-p_{k,i}|}{2}}(p_{k,i})$, we use $|x-z|\geq\frac12|x-p_{k,i}|$ to deduce
\begin{align}
\label{3.9}
\int_{\o_{k,i}\cap B_{\frac{|x-p_{k,i}|}{2}}(p_{k,i})}\frac{h_1^ke^{u_k}(z)}{|x-z|}\,\mathrm{d}z\leq\frac{C}{|x-p_{k,i}|}\,.
\end{align}
Using jointly (\ref{3.8}) and (\ref{3.9}) we conclude that
\begin{align}
\label{3.10}
\int_{\o_{k,i}}\frac{h_1^ke^{u_k}(z)}{|x-z|}\,\mathrm{d}z\leq \frac{C}{|x-p_{k,i}|}
\end{align}
and analog estimate holds for $-2u_k$. By the latter properties we readily get
$$\inf_{i=1,\cdots,m}|x-p_{k,i}||\nabla u_k(x)|\leq C.$$
Hence, we have the thesis of the lemma.
\end{proof}

\medskip

Let $p_i=\lim_{k\rightarrow\infty}p_{k,i}\in\O$ for $i=1,\cdots,m$ and let $\mathcal{S}=\{p_1,\cdots,p_m\}$ be the blow-up set. The latter result yields a uniform bound of the bubbling solution outside the blow-up set.
\begin{lem}
\label{le3.4}
$u_k$ is uniformly bounded in any compact subset of $\ov\O\setminus\mathcal{S}.$
\end{lem}

\begin{proof}
We choose $\e>0$ small enough such that the set $\o_{\e}=\O\setminus\bigcup_{i=1}^m B_{\e}(p_i)$ is connected. On the other hand, we note if $k$ is sufficiently large we have
$$\inf_{i=1,\cdots,m}|x-p_{k,i}|\geq\frac{\e}{2}, \qquad \forall x\in\o_{\e}.$$
Using Lemma \ref{le3.3} we have
$$|\nabla u_k|\leq C_\e \qquad \mathrm{in}~\o_{\e}.$$
By choosing some $\bar x\in\o_{\e}\cap\p\o$ we get that
$$|u_k(x)|=|u_k(x)-u_k(\bar x)|\leq C_\e, \qquad \forall x\in\o_{\e}.$$
Thus, we obtain the conclusion.
\end{proof}

\medskip

Furthermore, we collect in the following some standard information concerning the blow-up phenomenon.
\begin{lem}
\label{le3.5}
Let $u_k$ be a sequence of solutions to (\ref{eq2}). Then, it holds
\begin{align}
\label{3.12}
h_1^ke^{v_k}dx \rightharpoonup r_1(x)\,dx+\sum_{p\in\mathcal{S}\cap\o}m_1(p)\delta_{p} \qquad \mathrm{in}~\o,\\
\label{3.13}
h_2^ke^{w_k}dx \rightharpoonup r_2(x)\,dx+\sum_{p\in\mathcal{S}\cap\o}m_2(p)\delta_{p} \qquad \mathrm{in}~\o,
\end{align}
where $r_i(x)\in L^1(\O)\cap C_{\mathrm{loc}}^{\infty}(\O\setminus\mathcal{S})$, $i=1,2$, $m_1(p)\in 8\pi\N$ and $m_2(p)\in 4\pi\N$. Moreover, $u_k$ converges to $\mathcal{G}+\mathcal{U}$ in $C_{\mathrm{loc}}^\infty(\O\setminus\mathcal{S})$ and in $W_0^{1,q}(\o)$ for any $q<2$, where $\mathcal{G}$ and $\mathcal{U}$ satisfy
\begin{align*}
&\Delta \mathcal{G}(x)+\sum_{p\in\mathcal{S}\cap\o}(m_1(p)-m_2(p))\delta_p=0 \quad \mathrm{in}~\o, &\mathcal{G}(x)=0~\mathrm{on}~\p\o,\\
&\Delta \mathcal{U}(x)+r_1(x)-r_2(x)=0 \quad \mathrm{in}~\o, &\mathcal{U}(x)=0~\mathrm{on}~\p\o.
\end{align*}
\end{lem}

\begin{proof}
By minor modifications of the arguments in \cite[Theorem 3.4]{rz} we can get the convergence in (\ref{3.12})-(\ref{3.13}). Using the quantization result of Theorem \ref{th:quantization}, we know $m_1(p)\in 8\pi\N$ and $m_2(p)\in 4\pi\N$ when $p\in\o.$ We can complete the proof of the lemma by standard elliptic regularity theory.
\end{proof}

\medskip

We are now in a position to prove the main result, i.e. the exclusion of boundary blow-up.

\medskip

\noindent{\em Proof of Theorem \ref{th:boundary}.}
We have to prove that $S\cap\p\o=\emptyset$. Suppose it is not the case and take $x_0\in\mathcal{S}\cap\p\o.$ Taking $r$ small enough we may assume $\mathcal{S}\cap B_{r}(x_0)=\{x_0\}$. Consider then $z_k=x_0+\Theta_{k,r}\nu(x_0)$ with
\begin{equation}
\label{3.14}
\Theta_{k,r}=\frac{\int_{\p\o\cap B_r(x_0)}\le x-x_0,\nu\r\left|\frac{\p u_k}{\p\nu}\right|^2}{\int_{\p\o\cap B_r(x_0)}\le \nu(x_0),\nu\r\left|\frac{\p u_k}{\p\nu}\right|^2}\,,
\end{equation}
where $r$ is taken such that $\frac12\leq\le\nu(x_0),\nu\r\leq1$ for $x\in\p\o\cap B_r(x_0)$. $\nu(x)$ denotes the unit outer normal at $x\in\p\o$. Observe that $|\Theta_{k,r}|\leq 2r$ for $|\le x-x_0,\nu\r|\leq r$. Writing
$$x-z_k=x-x_0-\Theta_{k,r}\nu(x_0),$$
we deduce
\begin{align}
\label{3.15}
\int_{\p\o\cap B_r(x_0)}\le x-z_k,\nu\r\left|\frac{\p u_k}{\p \nu}\right|^2=0.
\end{align}
Using the Pohozaev identity (\ref{2.2}) in $\o\cap B_r(x_0)$ with $x_k$ replaced by $z_k,$ we get
\begin{align}
\label{3.16}
&\int_{\o\cap B_r(x_0)}\left(2h_1^ke^{u_k}+h_2^ke^{-2u_k}\right)
+\int_{\o\cap B_r(x_0)}\left(e^{u_k}\le x-z_k,\nabla h_1^k\r + \frac12e^{-2u_k}\le x-z_k,\nabla h_2^k\r\right)\nonumber\\
&=\int_{\p(\o\cap B_r(x_0))}\left(h_1^ke^{u_k}+\frac12h_2^ke^{-2u_k}\right)\le x-z_k,\nu\r+\int_{\p(\o\cap B_r(x_0))}\frac{\p u_k}{\p\nu}\le x-z_k,\nabla u_k\r\\
&\quad\quad-\frac12\int_{\p(\o\cap B_r(x_0))}|\nabla u_k|^2\le x-z_k,\nu\r.\nonumber
\end{align}
By the Dirichlet boundary conditions one see that
\begin{align*}
\lim_{k\rightarrow+\infty}\int_{\p\o\cap B_r(x_0)}\left(h_1^ke^{u_k}+h_2^ke^{-2u_k}\right)\le x-z_k,\nu\r=O(r^2),
\end{align*}
Moreover, by (\ref{3.15}) we get
\begin{align*}
&\int_{\p\o\cap B_r(x_0)}\frac{\p u_k}{\p\nu}\le x-z_k,\nabla u_k\r-\frac12\int_{\p\o\cap B_r(x_0)}|\nabla u_k|^2\le x-z_k,\nu\r\\
=~&\frac12\int_{\p\o\cap B_r(x_0)}\le x-z_k,\nu\r|\nabla u_k|^2=0.
\end{align*}
From the total energy bound and the assumptions on $h_i^k$, see \eqref{a2}, \eqref{a3}, we readily have
\begin{align*}
\lim_{k\rightarrow+\infty}\int_{\o\cap B_r(x_0)}e^{u_k}\le x-z_k,\nabla h_1^k\r=O(r)
\end{align*}
and
\begin{align*}
\lim_{k\rightarrow+\infty}\int_{\o\cap B_r(x_0)}e^{-2u_k}\le x-z_k,\nabla h_2^k\r=O(r).
\end{align*}
The left terms in \eqref{3.16} can be estimated as follows.
\medskip

\no\textbf{Claim:}
\begin{align*}
\lim_{k\rightarrow+\infty}\int_{\o\cap\p B_r(x_0)}\left(h_1^ke^{u_k}+\frac12h_2^ke^{-2u_k}\right)\le x-z_k,\nu\r=O(\e(r)),
\end{align*}
and
\begin{align*}
\int_{\o\cap\p B_r(x_0)}\frac{\p u_k}{\p\nu}\le x-z_k,\nabla u_k\r-\frac12\int_{\o\cap\p B_r(x_0)}|\nabla u_k|^2\le x-z_k,\nu\r=O(\e(r)),
\end{align*}
where $\e(r)\rightarrow0$ as $r\rightarrow0$. We postpone the proof of this estimates in the next lemma.

\medskip

Returning to the Pohozaev identity in \eqref{3.16}, by all the previous estimates we conclude
\begin{align*}
\lim_{r\rightarrow0}\lim_{k\rightarrow\infty}\int_{\o\cap B_r(x_0)}\left(2h_1^ke^{u_{k}}+h_2^ke^{-2u_k}\right)=0,
\end{align*}
which contradicts the minimal energy stated in (\ref{3.3}). The proof is concluded once we get the claim.
\begin{flushright}
$\square$
\end{flushright}

\begin{lem}
\label{le3.7}
For any $\e$ there exists $r=r(\e)$ such that
\begin{equation}
\label{3.17}
\lim_{k\rightarrow\infty}\int_{\o\cap\p B_r(x_0)}\left(h_1^ke^{u_k}+\frac12h_2^ke^{-2u_k}\right)|\le x-z_k,\nu\r|=O(\e),
\end{equation}
and
\begin{align}
\label{3.18}
\int_{\o\cap\p B_r(x_0)}\left|\frac{\p u_k}{\p\nu}\le x-z_k,\nabla u_k\r-\frac12|\nabla u_k|^2\le x-z_k,\nu\r\right|=O(\e).
\end{align}
\end{lem}

\begin{proof}
We start by recalling that from Lemma (\ref{le3.5}) we have $u_k\rightarrow\mathcal{G}+\mathcal{U}$ in $C_{\mathrm{loc}}^\infty(\o\setminus\mathcal{S})$. Take now $r\in\left(0,\frac12\mathrm{dist}(x_0,\mathcal{S}\setminus\{x_0\})\right)$. Then, $\|\mathcal{G}\|_{C^2(\o\cap\p B_r(x_0))}\leq C$ on $\o\cap\p B_{r}(x_0),$ for some $C$ independent of $r$. Moreover, observe that
$$|x-z_k|=|x-x_0-\Theta_{k,r}\nu(x_0)|\leq|x-x_0|+|\Theta_{k,r}|=O(r), \qquad \mathrm{for}~x\in\p B_{r}(x_0)\cap\o.$$
By the latter estimates, the conclusion of the lemma will follow by showing that for any $\e$ there exists $r=r(\e)$ such that
\begin{align}
\label{3.19}
\begin{split}
&\int_{\o\cap\p B_r(x_0)}re^{2|\mathcal{U}|}=O(\e), \\
&\int_{\o\cap\p B_r(x_0)}r|\nabla \mathcal{U}|^2=O(\e).
\end{split}
\end{align}
The proof of the above two equalities are almost the same and we only give the details of the second one. Recalling that
$$\Delta \mathcal{U}(x)+r_1(x)-r_2(x)=0 ~\mathrm{in}~\o,\qquad \mathcal{U}(x)=0~\mathrm{on}~\p\o,$$
for any $x\in \p B_r(x_0)\cap\o$ we have by the representation formula
\begin{align}
\label{3.20}
|\nabla \mathcal{U}(x)|=&\int_{\o}\left|\nabla G(x,y)\bigr(r_1(y)-r_2(y)\bigr)\right|\,\mathrm{d}y\nonumber\\
\leq&~C\int_{\o}\frac{1}{|x-y|}\bigr(|r_1(y)|+|r_2(y)|\bigr)\,\mathrm{d}y\\
\leq&~C\int_{\o\cap B_{r'}(x_0)}\frac{1}{|x-y|}\bigr(|r_1(y)|+|r_2(y)|\bigr)\,\mathrm{d}y+C\int_{\o\setminus B_{r'}(x_0)}\frac{1}{|x-y|}\bigr(|r_1(y)|+|r_2(y)|\bigr)\,\mathrm{d}y\nonumber,
\end{align}
with $r'$ such that $B_{3r'}(x_0)\cap(\mathcal{S}\setminus\{x_0\})=\emptyset$ and
\begin{align*}
\int_{\o\cap B_{r'}(x_0)}\bigr(|r_1(y)|+|r_2(y)|\bigr)\mathrm{d}y\leq\delta
\end{align*}
where $\delta$ will be determined later. Observe that $r'$, $\delta$ and the constants $C$ in (\ref{3.20}) are independent of $r$. For the integral outside $B_{r'}(x_0)$ in (\ref{3.20}) we have
\begin{align}
\label{3.21}
\int_{\o\setminus B_{r'}(x_0)}\frac{1}{|x-y|}\bigr(|r_1(y)|+|r_2(y)|\bigr)\,\mathrm{d}y\leq C\frac{1}{r'}\,,
\end{align}
where $C=C\left(\|r_i\|_{L^1(\o)},\o\right),~i=1,2$. For the other term in \eqref{3.20} we consider the following:
\begin{align*}
&\int_{\o\cap B_{r'}(x_0)}\frac{1}{|x-y|}\bigr(|r_1(y)|+|r_2(y)|\bigr)\,\mathrm{d}y\\
=&\int_{(\o\cap B_{r'}(x_0))\cap B_{\frac{r}{N}}(x)}\frac{1}{|x-y|}\bigr(|r_1(y)|+|r_2(y)|\bigr)\,\mathrm{d}y\\
&+\int_{(\o\cap B_{r'}(x_0))\setminus B_{\frac{r}{N}}(x)}\frac{1}{|x-y|}\bigr(|r_1(y)|+|r_2(y)|\bigr)\,\mathrm{d}y\\
=&~I_1+I_2,
\end{align*}
where $N$ will be suitably chosen later.

We start by estimating $I_1$. By Lemma \ref{le3.3} and using $B_{3r'}(x_0)\cap(\mathcal{S}\setminus\{x_0\})=\emptyset,$ we get
\begin{equation*}
|x-x_0|^2\max\bigr\{|r_1(x)|,|r_2(x)|\bigr\}\leq C, \qquad \mathrm{in}~\o\cap B_{r'}(x_0).
\end{equation*}
We may further assume $r<\min\bigr\{\frac{r'}{4},\frac12\mathrm{dist}\bigr(x_0,\mathcal{S}\setminus\{x_0\}\bigr)\bigr\}$. Observe that
$$|y-x_0|\geq|x-x_0|-|x-y|=\frac{N-1}{N}r, \qquad y\in B_{\frac{r}{N}}(x).$$
It follows that
\begin{equation*}
\max\bigr\{|r_1(y)|,|r_2(y)|\bigr\}\leq C\Big(\frac{N}{N-1}\Big)^2\frac{1}{r^2}\,.
\end{equation*}
Therefore, we deduce
\begin{align}
\label{3.24}
I_1\leq C\Big(\frac{N}{N-1}\Big)^2\frac{1}{r^2}\frac{r}{N}
\leq C\frac{1}{N}\frac{1}{r}\,.
\end{align}
For what concerns $I_2$, observing that $\frac{1}{|x-y|}\leq\frac{N}{r}$ for $y\in\bigr(\o\cap B_{r'}(x_0)\bigr)\setminus B_{\frac{r}{N}}(x)$ we get
\begin{align}
\label{3.25}
I_2\leq C\frac{N}{r}\int_{\o\cap B_{r'}(x_0)}\bigr(|r_1(y)|+|r_2(y)|\bigr)\,\mathrm{d}y\leq C\delta\frac{N}{r}\,.
\end{align}
From (\ref{3.20})-(\ref{3.25}), we deduce
$$|\nabla\mathcal{U}(x)|\leq C\frac{1}{r'}+C\left(\frac{1}{N}+\delta N\right)\frac{1}{r}\,.$$
To conclude we have to determine $N,r'$ and $r$. We start by choosing $N$ sufficiently large and then $r'$ small such that $C\big(\frac{1}{N}+\delta N\big)<\varepsilon.$ Note that the choices of $N$ and $r'$ are independent of $r$. Finally, $r$ is taken sufficiently small such that
\begin{align*}
C\frac{r^2}{r'^{\hspace{0.01cm}2}}+C\left(\frac{1}{N}+\delta N\right)\leq C\e.
\end{align*}
This conclude the second estimate in (\ref{3.19}). The argument for the first estimate is very similar. Hence, we prove the lemma and the claim in the proof of Theorem \ref{th:boundary}.
\end{proof}

\

\section{The Moser-Trudinger inequality} \label{sec:m-t}

\medskip

We are concerned now with the equation \eqref{eq3} defined on a compact surface. In this section we give a proof of the sharp Moser-Trudinger inequality related to this problem, see Theorem \ref{th:m-t}. The argument is mainly based on the blow-up analysis and it relies on the quantization result obtained in Theorem \ref{th:quantization}. On can reason similarly as in \cite{os}. We follow the strategy introduced by W. Ding in \cite{ding} for the standard Moser-Trudinger inequality and then used in \cite{bat-mal, jost-wang} for the inequality related to the Toda systems. For what concerns the optimal inequality we mainly follow the argument in \cite{bat-mal}. Such inequality was derived also in \cite{ri-su} from a dual point of view in the framework of equations involving probability measures.

We start by giving a description of the blow-up phenomenon, then we first prove a partial result concerning Theorem \ref{th:m-t} by introducing a modified functional and in a second step we complete the proof of the sharp result.

\medskip

\subsection{Preliminaries}

For a sequence of solutions $u_k$ of \eqref{eq3} relative to $\rho_{i,k} \to \bar \rho_i$, $i=1,2$, we consider the normalized functions
\begin{align*}
&	u_{1,k} = u_k - \log \int_M h_{1} \,e^{u_{k}} \,dV_g, \\
&	u_{2,k} = -2u_k - \log \int_M h_{2} \,e^{-2u_{k}} \,dV_g,
\end{align*}
that satisfy
$$
	-\D u_{1,k} = \rho_{1,k} \left(h_1 e^{u_{1,k}}-1\right) - \rho_{2,k} \left( h_2 e^{u_{2,k}}-1\right).
$$
Observe that
$$
	\int_M h_i e^{u_{i,k}} \,dV_g = 1, \quad i=1,2.
$$
We define the blow-up sets to be
\begin{equation} \label{blow-up set}
S_i= \biggr\{ p\in M \,:\, \exists\{x_k\}\subset M, \, x_k\to p, \, u_{i,k}(x_k) \to +\infty \biggr\}, \quad i=1,2.
\end{equation}
We point out that the argument for the quantization property in Theorem \ref{th:quantization} can be adapted to the above equation: it is then standard to get the following alternative, see for example \cite{jwy2,os} (see also Lemma \ref{le3.5}). We point out that all the following results hold true up to adding suitable constants.
\begin{theorem} \label{th:blow-up}
Let $u_k$ be a sequence of solutions to \eqref{eq3} relative to $\rho_{i,k} \to \bar \rho_i$, $i=1,2$, and let $S=S_1\cup S_2$ where $S_i$, $i=1,2,$ are defined in \eqref{blow-up set}. Then, up to subsequences, the following alternative holds true:
\begin{enumerate}
	\item (compactness) $S=\emptyset$ and $u_k$ is uniformly bounded in $L^{\infty}(M)$.
	
	\item (blow-up) $S\neq\emptyset$ and it is finite. It holds
	$$
		\rho_{i,k} h_i e^{u_{i,k}} \rightharpoonup r_i + \sum_{p\in S_i} m_i(p) \d_p, \quad i=1,2,
	$$
	in the sense of measures, where $r_i \in L^1(M)\cap L^{\infty}_{loc}(M\setminus S_i)$ and
	$$
		m_i(p) = \lim_{r\to 0} \lim_{k\to+\infty} \rho_{i,k} \int_{B_r(p)} h_i e^{u_{i,k}} \,dV_g.
	$$
	Moreover, $m_1(p)\in8\pi\N$, $m_2(p)\in4\pi\N$ and $r_i=0$ for some $i=1,2$.
\end{enumerate}
\end{theorem}

\medskip

\no We state now some important corollaries that will be used later on in the existence problem. A direct consequence of the latter result is the following compactness property.
\begin{cor} \label{cmpt}
Let $\Lambda=(8\pi\N \times \R) \cup (\R \times 4\pi\N)$. Suppose $\rho=(\rho_1,\rho_2)$ are in a fixed compact set of $\R^2 \setminus \L$. Then, the set of solutions $\{u_\rho\}_\rho$ is uniformly bounded in $L^\infty(M)$.
\end{cor}

\medskip

The latter uniform bound implies that one can take a high sublevel $J_\rho^L$ containing all the critical points of the functional. Then we can deform the whole space $H^1(M)$ onto this sublevel just by following a gradient flow to get the following result.
\begin{cor} \label{high}
Suppose $\rho=(\rho_1,\rho_2)\notin \Lambda$. Then, for some large $L>0$, $J_\rho^L$ is a deformation retract of $H^1(M)$ and in particular it is contractible.
\end{cor}

\medskip

Moreover, exploiting the compactness result one can suitably adapt the argument in \cite{lucia} to get the following useful topological argument.
\begin{cor} \label{top-arg}
Let $a,b\in\R$ be such that $a<b$ and $J_\rho$ has no critical points $u\in H^1(M)$ with $a\leq J_\rho(u) \leq b$. Suppose $\rho=(\rho_1,\rho_2)\notin \Lambda$. Then, $J_\rho^a$ is a deformation retract of $J_\rho^b$.
\end{cor}

\medskip

\subsection{The Moser-Trudinger inequality}

We introduce now the argument for proving the main result of this section, see Theorem \ref{th:m-t}.  We start by giving the following definition which will be then used in the sequel. We set
$$
 \mathcal B = \left\{ \rho=(\rho_1,\rho_2) \in [0,+\infty)^2 \,:\, \inf_{u\in H^1(M)} J_\rho(u) > -\infty  \right\}.
$$
Observe that $\mathcal B$ preserves a partial order in $[0,+\infty)^2$: more precisely, if $(\rho_1,\rho_2)\in\mathcal B$ then $(\rho_1',\rho_2')\in\mathcal B$ for $\rho_1'<\rho_1$ and $\rho_2'<\rho_2$. Moreover, by using standard scalar Moser-Trudinger inequalities for $u$ and $-2u$ respectively, it follows that $\mathcal B\neq \emptyset$.

Theorem \ref{th:m-t} can be then rephrased by asserting that
$$
	\mathcal B=[0,8\pi]\times[0,4\pi].
$$
Observe that we readily have $\mathcal B\subset[0,8\pi]\times[0,4\pi]$ by Proposition \ref{test-f}. We start now by proving a partial result which will be then used in the proof of the sharp result.
\begin{proposition} \label{prop:m-t}
It holds that $[0,8\pi)\times[0,4\pi)\subset \mathcal B$.
\end{proposition}

In order to prove the latter result we start by analyzing what happens in $\overset{\circ}{\mathcal B}$ and on $\p \mathcal B$.
\begin{lem} \label{l:set-B}
Let $\rho=(\rho_1,\rho_2)\in \overset{\circ}{\mathcal B}$, then $J_\rho$ has a minimizer $u\in H^1(M)$ which solves \eqref{eq3}. If instead $\rho \in \p\mathcal B$, there exists a sequence $\{ u_k \}_k\subset H^1(M)$ such that
$$
	\lim_{k\to+\infty} \int_M |\n u_k|^2 \,dV_g = +\infty, \qquad \lim_{k\to+\infty} \frac{J_\rho(u_k)}{\int_M |\n u_k|^2 \,dV_g} 	\leq 0.
$$
\end{lem}

\begin{proof}
For the first part one has just to take $\d>0$ sufficiently small so that $(1+\d)\rho\in \mathcal B$ and to notice that
$$
	J_\rho (u) = \frac{\d}{2(1+\d)} \int_M |\n u|^2 \,dV_g + \frac{J_{(1+\d)\rho}(u)}{1+\d} \geq \frac{\d}{2(1+\d)} \int_M |\n u|^2 \,dV_g - C.
$$
Restricting ourselves to the zero average functions, the functional is coercive and weakly lower-semicontinuous, so we can minimize it.

\medskip

Concerning the second part, suppose by contradiction that for any $\{u_k\}_k$ with $\int_M |\n u_k|^2 \,dV_g \to +\infty$ we would have
$$
	\frac{J_\rho(u_k)}{\int_M |\n u_k|^2 \,dV_g} \geq \e > 0.
$$
This implies that $J_\rho (u) \geq \frac \e2 \int_M |\n u|^2 \,dV_g -C$ and we deduce that for $\d>0$ sufficiently small
$$
	J_{(1+\d)\rho} (u) = (1+\d)J_\rho(u) - \d \int_M |\n u|^2 \,dV_g \geq \wtilde \e \int_M |\n u|^2 \,dV_g - C \geq -C,
$$
so we conclude that $(1+\d)\rho \in \mathcal B$ which contradicts the assumption that $\rho \in \p\mathcal B$.
\end{proof}

\medskip

In proving Proposition \ref{prop:m-t} our aim is to exploit the blow-up analysis in Theorem \ref{th:blow-up}. To this end we perturb our functional $J_\rho$ to \emph{force} it to exhibit blow-up. We start by stating the following fact which can be found in \cite{ding,jost-wang}: for any two sequences $\{a_k\}_k$ and $\{b_k\}_k$ satisfying
$$
	\lim_{k\to+\infty} a_k = +\infty, \qquad \lim_{k\to+\infty} \frac{b_k}{a_k} 	\leq 0,
$$
there exists a smooth function $F:[0,+\infty)\to\R$ such that
\begin{equation} \label{F}
	F'(t)\in (0,1), \qquad \lim_{t\to+\infty} F'(t)=0, \qquad \lim_{k\to+\infty} \bigr(b_{n_k} - F(a_{n_k})\bigr) = -\infty,
\end{equation}
for some subsequence $\{n_k\}_k$. We apply the latter result to
\begin{equation} \label{choice}
a_k = \frac 12 \int_M |\n u_k|^2 \,dV_g \qquad \mbox{and} \qquad  b_k=J_\rho(u_k),
\end{equation}
where $\{u_k\}_k$ is the sequence found in Lemma \ref{l:set-B} and we define the perturbed functional by
\begin{equation} \label{perturb}
	\wtilde J_\rho (u) = J_\rho(u)- F\left( \frac 12 \int_M |\n u|^2 \,dV_g \right).
\end{equation}
We point out that $\wtilde J_\rho$ is defined in such a way that for $\rho\in\overset{\circ}{\mathcal B}$ it has a minimizer $u\in H^1(M)$ which solves equation \eqref{eq3} with
\begin{equation}\label{parameter}
	\wtilde\rho_i = \frac{1}{1-\mu(u)}\,\rho_i, \qquad \mu(u) = F'\left( \frac 12 \int_M |\n u|^2 \,dV_g \right).
\end{equation}
Indeed, it is possible to argue as in Lemma \ref{l:set-B} by exploiting the properties of the function $F$. Moreover, for $\rho\in\p\mathcal B$ one can use the sequence $\{u_k\}_k$ obtained in Lemma \ref{l:set-B} and the choice of $a_k, b_k$ in \eqref{choice} to deduce
\begin{equation} \label{inf}
	\inf_{u\in H^1(M)} \wtilde J_\rho (u) = -\infty.
\end{equation}
We are now in a position to prove Proposition \ref{prop:m-t}.

\medskip

\no \emph{Proof of Proposition \ref{prop:m-t}.}
We argue by contradiction. Suppose the thesis is false, then there exists $\ov\rho=(\ov\rho_1,\ov\rho_2)\in\p\mathcal B$ with $\ov\rho_1 < 8\pi$ and $\ov\rho_2 < 4\pi$. Take $\rho_k \in \overset{\circ}{\mathcal B}$ with $\rho_k \to \ov\rho$ and let $\{u_k\}_k$ be the associated minimizers of $\wtilde J_{\rho_k}$ defined in \eqref{perturb} satisfying equation \eqref{eq3} with parameter $\wtilde\rho$ given by \eqref{parameter}, see the argument above. We may suppose to work with zero average functions.

\medskip

Suppose first that
$$
\frac 12 \int_M |\n u_k|^2 \,dV_g \leq C,
$$
for some $C$ independent of $k$. Then, by Theorem \ref{th:blow-up} we would get the sequence $\{u_k\}_k$ admits a limit $u\in H^1(M)$ which is a minimizer of $\wtilde J_{\ov\rho}$. This is not possible by construction, see \eqref{inf}.

\medskip

We deduce that the sequence $\{u_k\}_k$ has to blow-up, as anticipated before. Recall that $\{u_k\}_k$ satisfy equation \eqref{eq3} with parameters $\wtilde\rho$ given by \eqref{parameter}. By construction, see \eqref{F}, we get $\mu(u_k)\to 0$ and hence $\wtilde \rho \to \ov\rho < (8\pi,4\pi)$. This contradicts the necessary condition for a blowing-up solution given by Theorem~\ref{th:blow-up}, see also Corollary \ref{cmpt}. The proof is concluded.
\begin{flushright}
$\square$
\end{flushright}

\subsection{The sharp inequality}

We are going to prove here the sharp inequality, namely Theorem \ref{th:m-t}. We start by pointing out a version of the standard Moser-Trudinger inequality in \eqref{ineq} on bounded domains with Dirichlet boundary condition: let $\Omega \subset \R^2$ be a bounded domain, then for any $v\in H^1_0(\Omega)$ it holds
\begin{equation} \label{ineq-dom}
		8\pi \log \int_\Omega e^{v} \,dV_g \leq \frac 12 \int_{\Omega} |\n v|^2 \,dx + C_\Omega.
\end{equation}
We will need both the inequality in \eqref{ineq} and a \emph{localized} version of it around a blow-up point, see the following result.
\begin{lem} \label{l:loc}
Let $p\in M$ be a blow-up point of $\{u_k\}_k$. Then, for all $\d>0$ small there exists $C_\d>0$ such that
$$
	8\pi \log \int_M e^{u_k-\ov u_k} \,dV_g \leq \frac 12 \int_{B_\d(p)} |\n u_k|^2 \,dV_g + C_\d.
$$
\end{lem}

\begin{proof}
We start by taking $\d>0$ small such that $B_\d(p)$ contains no other blow-up point. Furthermore, we may suppose that in $B_\d(p)$ we have a flat metric and $\ov u_k =0$. In order to use the inequality in \eqref{ineq-dom} we modify the function in the following way: let $v_k$ be the solution of
$$
\left\{
\begin{array}{rll}
-\D v_k =& 0 & \mbox{in } B_\d(p), \\
v_k =& u_k  & \mbox{on } \p B_\d(p).
\end{array}
\right.
$$
The latter auxiliary function is bounded: indeed, by elliptic estimates and by the estimates of $u_k$ outside the blow-up set, see for example Sections \ref{sec:quantization}, \ref{sec:boundary}, we get
$$
	\| v_k \|_{C^1(B_\d(p))} \leq C \| v_k \|_{L^{\infty}(B_\d(p))} \leq \| u_k \|_{L^{\infty}(\p B_\d(p))} \leq C,
$$
for some $C>0$ independent on $k$. We then set $\wtilde u_k = u_k - v_k$ so that $\wtilde u_k\in H^1_0(B_\d(p))$. Applying the Moser-Trudinger inequality for bounded domains with Dirichlet boundary condition \eqref{ineq-dom} we deduce
\begin{equation} \label{dom}
	8\pi \log \int_{B_\d(p)} e^{\wtilde u_k} \,dV_g \leq \frac 12 \int_{B_\d(p)} |\n \wtilde u_k|^2 \,dV_g + C_\d.
\end{equation}
By the same estimates on $v_k, u_k$ we first observe that the gradient terms of $\wtilde u_k$ and $u_k$ are different by an $O(1)$ term:
\begin{align}
	\int_{B_\d(p)} |\n \wtilde u_k|^2 \,dV_g &= \int_{B_\d(p)} |\n u_k|^2 \,dV_g + \int_{B_\d(p)} |\n v_k|^2 \,dV_g - 2\int_{B_\d(p)} \n u_k\cdot \n v_k \,dV_g \nonumber\\
	& \leq \int_{B_\d(p)} |\n u_k|^2 \,dV_g + \int_{B_\d(p)} |\n v_k|^2 \,dV_g + 2\|\n v_k\|_{L^{\infty}(B_\d(p))} \int_{B_\d(p)} |\n u_k| \,dV_g \nonumber\\
	& \leq \int_{B_\d(p)} |\n u_k|^2 \,dV_g + C. \label{dom1}
\end{align}
Concerning the nonlinear term, by the estimates on $v_k$ we have
$$
	\int_{B_\d(p)} e^{\wtilde u_k} \,dV_g = \int_{B_\d(p)} e^{u_k-v_k} \,dV_g \geq C \int_{B_\d(p)} e^{u_k} \,dV_g \geq C \theta \int_{M} e^{u_k} \,dV_g,
$$
where in the last inequality we used the fact that $u_k$ blows-up at $p$ and hence $\frac{e^{u_k}}{\int_M e^{u_k}\,dV_g} \rightharpoonup \d_p$, see for example Theorem \ref{th:blow-up}. Therefore, for $k$ big enough $\int_{B_\d(p)} e^{u_k} \,dV_g \geq \theta \int_{M} e^{u_k} \,dV_g$, for some $0<\theta<1$. It follows that
\begin{equation} \label{dom2}
\log \int_{B_\d(p)} e^{\wtilde u_k} \,dV_g \geq \log\int_{M} e^{u_k} \,dV_g - C_\theta.
\end{equation}
Inserting \eqref{dom1} and \eqref{dom2} into \eqref{dom} we get the thesis.
\end{proof}

\medskip

We proceed now with the proof of Theorem \ref{th:m-t}.

\medskip

\no \emph{Proof of Theorem \ref{th:m-t}.}
We have to show that $\mathcal B=[0,8\pi]\times[0,4\pi]$. To do this we take $\rho_k\in [0,8\pi)\times[0,4\pi)$ with $\rho_k \to (8\pi,4\pi)$ and we prove $\inf_{H^1(M)} J_{\rho_k} \geq -C$, for some $C>0$ independent on $k$. Since by Proposition \ref{prop:m-t} we already know $J_{\rho_k}$ has a minimizer $u_k$, if we show $J_{\rho_k}(u_k)\geq -C$ then the thesis follows. We may assume $\ov u_k = 0$.

\medskip

If the sequence $\{u_k\}_k$ does not blow-up, Theorem \ref{th:blow-up} yields it converges to a minimizer of $J_{(8\pi,4\pi)}$ and we are done. Hence, suppose the sequence of minimizers does blow-up: more precisely either $u_k$ or $-2u_k$ blow-up (or both). If both components blow-up at the same point, then the Pohozaev identity \eqref{2.3} holds true. The values of $\rho_k=(\rho_{1,k},\rho_{2,k})$ can not satisfy the latter identity so this situation can not happen.

\medskip

We are left with the following alternative: either $u_k$ blows-up at a point $p\in M$ and $-2u_k$ stays bounded (or vice versa) or $u_k$ and $-2u_k$ blow-up at different points $p_1,p_2\in M$. Suppose the first situation occurs. Since $\ov u_k = 0$ and $-2u_k$ is bounded we have
\begin{align*}
	J_{\rho_k}(u_k) &= \frac 12 \int_M |\n u_k|^2 \,dV_g - \rho_{1,k} \log \int_M e^{u_k} \,dV_g - \frac{\rho_{2,k}}{2} \log \int_M e^{-2u_k} \,dV_g \\
	 & \geq  \frac 12 \int_M |\n u_k|^2 \,dV_g - \rho_{1,k} \log \int_M e^{u_k} \,dV_g - C.
\end{align*}
By the fact that $\rho_{1,k}\to 8\pi$ we exploit the standard Moser-Trudinger inequality in \eqref{ineq} to assert that $J_{\rho_k}(u_k)>-C$, for some $C>0$ independent on $k$, which is the desired property.

\medskip

Suppose now both $u_k$ and $-2u_k$ blow-up at different points $p_1,p_2\in M$. For $\d>0$ sufficiently small we have
\begin{align*}
	J_{\rho_k}(u_k) &\!=\! \frac 12 \int_M |\n u_k|^2 \,dV_g - \rho_{1,k} \log \int_M e^{u_k} \,dV_g - \frac{\rho_{2,k}}{2} \log \int_M e^{-2u_k} \,dV_g \\
	 & \!\geq\!  \left(\frac 12 \int_{B_\d(p_1)} \!\!\! |\n u_k|^2 \,dV_g - \rho_{1,k} \log \int_M e^{u_k} \,dV_g\right)\! + \! \left(\frac 12 \int_{B_\d(p_2)} \!\!\! |\n u_k|^2 \,dV_g - \frac{\rho_{2,k}}{2} \log \int_M e^{-2u_k} \,dV_g\right)\!.
\end{align*}
Then, we apply the local Moser-Trudinger inequality of Lemma \ref{l:loc} to both $u_k$ and $-2u_k$ around $p_1$ and $p_2$, respectively. One has just to observe that $(\rho_{1,k},\rho_{2,k})\to (8\pi,4\pi)$ and that the scaling involving the part with $-2u_k$ gives a sharp constant of $\rho_2=4\pi$ to conclude that $J_{\rho_k}(u_k)>-C$, for some $C>0$ independent on $k$. This concludes the proof of the main theorem.
\begin{flushright}
$\square$
\end{flushright}

\medskip

\section{A general existence result} \label{sec:existence}

\medskip

In this section we introduce the variational argument to prove the general existence result stated in Theorem \ref{th:existence}. The plan is the following: we start by getting an improved Moser-Trudinger inequality and by describing the topological set one should consider, next we construct the test functions modeled on the latter set and finally we prove the existence of solutions via a topological argument. We mainly follow the argument in \cite{bjmr}: when we will be sketchy we refer to the latter paper for the full details.

\medskip

\subsection{Improved Moser-Trudinger inequality and the topological join} \label{subs:m-t}

In this subsection we obtain an improved Moser-Trudinger inequality \eqref{m-t} and we show how can be used it in the study of the low sublevels of the functional $J_\rho$. More precisely, if $e^u$ and $e^{-2u}$ are spread in different regions over the surface then the constants in \eqref{m-t} can be multiplied by some positive integers. This in turn implies that in the very negative sublevels of $J_\rho$, $e^u$ or $e^{-2u}$ have to concentrate around a finite number of points. In this way one can map these configurations onto the topological join as discussed in the Introduction.

\medskip

Before giving the improved inequality we state a simple lemma the proof of which can be found for example in \cite{bjmr}.
\begin{lem}\label{l:step1}
Let $\dis{\d>0}$, $\dis{\th>0}$, $\dis{k,l\in\N}$ with $\dis{k\ge l}$, $\dis{f_i\in L^1(M)}$ be non-negative functions with $\dis{\|f_i\|_{L^1(M)}=1}$ for $\dis{i=1,2}$ and $\dis{\{\O_{1,i},\O_{2,j}\}_{i\in\{1,\dots,k\},j\in\{1,\dots,l\}}\subset M}$ such that
$$d(\O_{1,i},\O_{1,i'})\ge\d, \quad d(\O_{2,j},\O_{2,j'}) \ge\d,\qquad\forall\;i, \ i'\in\{1,\dots,k\}\hbox{ with }i\ne i' \;, \forall\;j,\ j' \in\{1,\dots,l\}\hbox{ with }j\ne j'$$
and
$$\int_{\O_{1,i}}f_1\,dV_g\ge\th, \; \int_{\O_{2,j}}f_2\,dV_g\ge\th, \qquad\forall\;i\in\{0,\dots,k\}, \; \forall\;j\in\{0,\dots,l\}.$$
Then, there exist $\dis{\ov\d>0,\;\ov\th>0}$, independent of
$\dis{f_i}$, and $\dis{\{\O_n\}_{n=1}^k\subset M}$ such that
$$d(\O_n,\O_{n'})\ge\ov\d,\qquad\forall\;n,\ n'\in\{1,\dots,k\}\hbox{ with }n\ne n'$$
and
$$\int_{\O_n}f_1\,dV_g\ge\ov\th, \quad \int_{\O_m}f_2\,dV_g \ge\ov\th,\qquad\forall\;n\in\{1,\dots,k\}, \; \forall\;m\in\{1,\dots,l\}.$$
\end{lem}

\medskip

The main result of this subsection is the following.
\begin{pro}\label{mt-impr}
Let $\dis{\d>0}$, $\dis{\th>0}$, $\dis{k,l\in\N}$ and
$\dis{\{\O_{1,i},\O
_{2,j}\}_{i\in\{1,\dots,k\},j\in\{1,\dots,l\}}\subset M}$ be
as in Lemma \ref{l:step1}. Then, for any $\dis{\e>0}$ there exists $\dis{C=C\left(\e,\d,\th,k,l,M\right)}$ such that if $u\in H^1(M)$ satisfies
\begin{align*}
\int_{\O_{1,i}} e^{u}\,dV_g\ge\th\int_M e^{u}\,dV_g,~\forall i\in\{1,\dots,k\},
\qquad \int_{\O_{2,j}} e^{-2u}\,dV_g\ge\th\int_M e^{-2u}\,dV_g,~\forall j\in\{1,\dots,l\},
\end{align*}
it follows that
$$8k\pi\log\int_M e^{u-\ov{u}}\,dV_g+\frac{4l\pi}{2}\log\int_M e^{-2(u-\ov u)}\,dV_g\leq \frac{1+\e}{2}\int_M |\n u|^2\,dV_g+C. $$
\end{pro}

\begin{proof}
We may assume $k\geq l$ and $\ov u=0.$ Letting $f_1=\frac{e^{u}}{\int_M e^{u}\,dV_g}$, $f_2=\frac{e^{-2u}}{\int_M e^{-2u}\,dV_g}$ we apply Lemma \ref{l:step1} to find $\dis{\ov\d>0,\;\ov\th>0}$ and $\dis{\{\O_n\}_{n=1}^k\subset M}$ such that
$$d(\O_n,\O_{n'})\ge\ov\d\qquad\qquad\forall\;n,\ n'\in\{1,\dots,k\}\hbox{ with }n\ne n',$$
\begin{equation} 	\label{vol}
\int_{\O_n}e^{u}\,dV_g\ge\ov\th {\int_M e^{u}\,dV_g}, \quad \int_{\O_m}e^{-2u}\,dV_g \ge\ov\th\int_M e^{-2u}\,dV_g, \qquad\forall\;n\in\{1,\dots,k\}, \; \forall\;m\in\{1,\dots,l\}.
\end{equation}
We then introduce $k$ cut-off functions $0\leq\chi_n\leq 1$ such that
\begin{equation} \label{cut}
	{\chi_n}_{|\O_n} \equiv 1, \quad {\chi_n}_{|M\setminus (B_{\ov\d/2}(\O_n))}\equiv 0, \quad |\n \chi_n| \leq C_{\ov \d}, \qquad n=1,\dots,k.
\end{equation}
At this point we decompose $u$ such that $u=v+w$, with $\ov v = \ov w = 0$ and $v\in L^{\infty}(M)$. Such decomposition will be suitably chosen later on. Since $v$ will be under control, our aim is to apply localized (in $\O_n$) Moser-Trudinger inequalities \eqref{m-t} to $w$, namely to $\chi_n w$. We start by observing that using the volume spreading of $u$, see \eqref{vol}, we have
\begin{align}
	\log \int_M e^u \, dV_g &\leq \log \int_{\O_n} e^u \,dV_g + C_{\ov\d} =  \log \int_{\O_n} e^{v+w} \,dV_g + C \leq \log \int_{\O_n} e^{w} \,dV_g + \|v\|_{L^{\infty}(M)}+ C \nonumber\\
	& \leq \log \int_{M} e^{\chi_n w} \,dV_g + \|v\|_{L^{\infty}(M)}+ C. \label{dec}
\end{align}
The same holds true for $-2u$: summing together and using the Moser-Trudinger inequality \eqref{m-t} we end up with
\begin{align*}
	8\pi \log \int_M e^u\,dV_g + \frac{4\pi}{2} \log \int_M e^{-2u}\,dV_g  \leq 8\pi \log \int_M e^{\chi_n w}\,dV_g + \frac{4\pi}{2} \log \int_M e^{-2\chi_n w}\,dV_g + C\|v\|_{L^{\infty}(M)}+ C
\end{align*}
\begin{equation} \label{mt1}
\leq \frac 12 \int_M |\n(\chi_n w)|^2 \,dV_g +8\pi\int_M \chi_n w \,dV_g - \frac{4\pi}{2}\int_M 2\chi_n w \,dV_g  + C\|v\|_{L^{\infty}(M)}+ C,
\end{equation}
for $n=1,\dots,l$. We have to analyze the latter terms.

\medskip

We start by observing that $\ov w = 0$: by the Poincar\'e's inequality and by the Young's inequality
\begin{equation} \label{mt2}
 \int_M \chi_n w \,dV_g \leq \int_M |w| \,dV_g \leq C \| w \|_{L^2(M)} \leq C \left( \int_M |\n w|^2 \,dV_g \right)^{\frac 12} \leq \e \int_M |\n w|^2 \,dV_g + C_{\e}.
\end{equation}

\medskip

Concerning the gradient term we have, by using the Young's inequality and recalling the construction of the cut-off functions in \eqref{cut}
\begin{align}
	\int_M |\n(\chi_n w)|^2 \,dV_g &= \int_M \biggr( \chi_n^2|\n w|^2 + |\n\chi_n|^2 w^2 + 2(\chi_n \n w)\cdot(w\n\chi_n) \biggr) \,dV_g \nonumber\\
	& \leq (1+\e)\int_M \chi_n^2 |\n w|^2 \,dV_g + \left(1+\frac{1}{\e}\right)\int_M w^2|\chi_n|^2 \,dV_g \nonumber\\
	& \leq (1+\e)\int_{B_{\ov\d/2}(\O_n)}  |\n w|^2 \,dV_g + C_{\e,\ov\d}\int_M w^2 \,dV_g. \label{mt3}
\end{align}

\medskip

Putting together \eqref{mt1}, \eqref{mt2} and \eqref{mt3} we deduce
\begin{align} \label{mt-f}
\begin{split}
8\pi \log \int_M e^u\,dV_g + \frac{4\pi}{2} \log \int_M e^{-2u}\,dV_g  \leq &\frac{1+\e}{2}\int_{B_{\ov\d/2}(\O_n)}  |\n w|^2 \,dV_g + \e \int_M |\n w|^2 \,dV_g\\
&+ C_{\e,\ov\d}\int_M w^2 \,dV_g + C\|v\|_{L^{\infty}(M)}+ C,
\end{split}
\end{align}
with $n=1,\dots,l$. For $m=l+1,\dots,k$ we can just use the spreading of $u$, see \eqref{vol}: proceeding as in \eqref{dec}, using the standard Moser-Trudinger inequality \eqref{ineq} for $\chi_m w$ and by the estimates \eqref{mt2}, \eqref{mt3}, we obtain
\begin{equation} \label{mt-f2}
8\pi \log \int_M e^u\,dV_g  \leq \frac{1+\e}{2}\int_{B_{\ov\d/2}(\O_m)}  |\n w|^2 \,dV_g + \e \int_M |\n w|^2 + C_{\e,\ov\d}\int_M w^2 \,dV_g + C\|v\|_{L^{\infty}(M)}+ C,
\end{equation}
for $m=l+1,\dots,k$. Summing up \eqref{mt-f}, \eqref{mt-f} and recalling that the sets $\O_j$ are disjoint we end up with
\begin{align}
8k\pi \log \int_M e^u\,dV_g + \frac{4l\pi}{2} \log \int_M e^{-2u}\,dV_g  \leq &\frac{1+\e}{2}\int_M  |\n w|^2 \,dV_g + \e \int_M |\n w|^2 \,dV_g \nonumber\\
&+ C_{\e,\ov\d}\int_M w^2 \,dV_g + C\|v\|_{L^{\infty}(M)}+ C. \label{mt-f3}
\end{align}

\medskip

Finally, we have to suitably choose $v,w$ to estimate the left terms. To this end, consider a basis of eigenfunctions of $-\D$ in $H^1(M)$ with zero average condition relative to positive non-decreasing eigenvalues $\{\l_j\}_{j\in\N}$. Let $N=N_{\e,\ov\d}=\max\left\{ j\in \N \,:\, \l_j < \frac{C_{\e,\ov\d}}{\e} \right\}$, where $C_{\e,\ov\d}$ is the constant in \eqref{mt-f3}. We set
$$
v=P_{E_{\l_N}}(u), \qquad w= P_{E_{\l_N}^\bot}(u),
$$
where $E_{\l_N}$ is the direct sum of the eigenspaces with eigenvalues less or equal than $\l_N$ and $P$ denotes the projection. It follows that
\begin{equation*}
C_{\e,\ov\d}\int_M w^2 \,dV_g \leq \frac{C_{\e,\ov\d}}{\l_N} \int_M |\n w|^2 \,dV_g \leq \e \int_M |\n w|^2 \,dV_g \leq \e \int_M |\n u|^2 \,dV_g.
\end{equation*}
Moreover, $v\in E_{\l_N}$ which is a finite-dimensional space and hence, recalling that $\ov v=0$, its $L^{\infty}$ and $H^1$ norms are equivalent, i.e.
\begin{equation*}
\| v \|_{L^{\infty}(M)} \leq C \left( \int_M |\n v|^2 \,dV_g \right)^{\frac 12} \leq \e \int_M |\n v|^2 \,dV_g + C_{\e} \leq \e \int_M |\n u|^2 \,dV_g + C_{\e},
\end{equation*}
where we used also the Young's inequality. By plugging the latter two estimates in \eqref{mt-f3} we deduce the thesis.
\end{proof}

\medskip

From the latter result we deduce by standard arguments that if the energy $J_\rho(u)$ is large negative at least one of the two terms $e^u$, $e^{-2u}$ has to concentrate around some points of the surface. Recall the definitions of $M_k$ in \eqref{M_k}, $J_\rho^a$ in \eqref{sub} and $\dkr$ in \eqref{dist}. We have (see \cite{bjmr}):
\begin{pro}\label{p:altern}
Suppose $\dis{\rh_1\in(8k\pi,8(k+1)\pi)}$ and
$\dis{\rh_2\in(4l\pi,4(l+1)\pi)}$. Then, for any
$\displaystyle{\e>0}$, there exists $\dis{L>0}$ such that any
$\dis{u\in J_\rh^{-L}}$ verifies either
$$\dkr\left(\frac{\
h_1e^{u}}{\int_M
h_1e^{u}\,dV_g},M_k\right)<\e\qquad\qquad\text{or}\qquad\qquad
\dkr\left(\frac{ h_2e^{-2u}}{\int_M
h_2e^{-2u}\,dV_g},M_l\right)<\e.$$
\end{pro}

\medskip

It follows that we can map continuously these configurations onto the sets $M_k$ in \eqref{M_k} by using the following know result, see \cite{bjmr} for a short proof of it.

\begin{lem}\label{l:projbar}
Given $j\in\N$, for $\e_0>0 $ sufficiently small there exists a
continuous retraction:
$$\psi_j: \bigr\{ \s \in \M(M),\ \dkr(\s , M_j) < \e_0 \bigr\} \to M_j.$$
In particular,
if $\s_n \rightharpoonup \s $ in the sense of measures, with $\s
\in M_j$, then $\psi_j(\s_n) \to \s$.
\end{lem}

\medskip

The alternative in Proposition \ref{p:altern} can be expressed in terms of the topological join $M_k*M_l$, see \eqref{join} and the discussion in the Introduction. Moreover, we can restrict ourselves to targets in a simpler subset by exploiting the assumption on the surface to have positive genus. Indeed, the next topological fact holds true.
\begin{lem} \label{new} Let $M$ be a compact surface with positive genus $g(M)>0$. Then, there exist two
simple closed curves $\gamma_1, \gamma_2 \subseteq M$ such that
\begin{enumerate}
\item $\gamma_1, \gamma_2$ do not intersect
each other;

\item there exist global retractions $\Pi_i: M \to \gamma_i$,
$i=1,2$.

\end{enumerate}
\end{lem}

\medskip

Finally, by collecting all these results we are in a position to get the following mapping.
\begin{pro}\label{p:map}
Suppose $\dis{\rh_1\in(8k\pi,8(k+1)\pi)}$ and
$\dis{\rh_2\in(4l\pi,4(l+1)\pi)}$. Then for $L$ sufficiently large there exists a continuous map
$$
	\Psi: J_\rho^{-L} \to (\gamma_1)_k*(\gamma_2)_l.
$$
\end{pro}

\begin{proof}
The proof can be obtained reasoning as in \cite{bjmr} (see also \cite{jkm}): we repeat it here for the reader's convenience. Recall the maps $\psi_j$ introduced in Lemma \ref{l:projbar}. By Proposition \ref{p:altern} we know that by taking $L$ sufficiently large either $\psi_k\left(\frac{h_1 e^{u}}{\int_M h_1 e^{u}\,dV_g}\right)$ or $\psi_l\left(\frac{h_2 e^{-2u}}{\int_M h_2 e^{-2u}\,dV_g}\right)$ is well-defined. We then set $d_1 = \dkr\left(\frac{{h}_1 e^{u}}{\int_M {h}_1 e^{u}\,dV_g}, M_k\right)$, $d_2 = \dkr\left(\frac{{h}_2 e^{-2u}}{\int_M {h}_2 e^{-2u}\,dV_g}, M_l\right)$ on which the join parameter will depend:
\begin{equation} \label{def:r}
s=s(d_1,d_2)=f\left (\frac{d_1}{d_1+d_2} \right),
 \qquad
f(x)= \left \{ \begin{array}{ll} 0 & \mbox{ if } x \in[0,1/4],\\
2z-\frac{1}2 & \mbox{ if } x \in (1/4, 3/4),\\
1 & \mbox{ if } x \in [3/4,1].\end{array} \right.
\end{equation}
Recall now the retractions $\Pi_i : \Sigma \to \gamma_i$, $i=1,2$ constructed in Lemma \ref{new}. Finally, we define the map
\begin{equation}
\label{eq:psi}
  \Psi(u) = (1-s) (\Pi_1)_* \psi_k \left(\frac{h_1 e^{u}}{\int_M {h}_1 e^{u}\,dV_g}\right)+ s\, (\Pi_2)_* \psi_l \left(\frac{{h}_2
  e^{-2u}}{\int_M h_2 e^{-2u}\,dV_g}\right),
\end{equation}
where $(\Pi_i)_*$ stands for the push-forward of the map $\Pi_i$.
Observe that when one of the two $\psi$'s is not defined the other necessarily
is so it is well-defined in the topological join $(\gamma_1)_k*(\gamma_2)_l$.
\end{proof}

\medskip

\subsection{Test functions} \label{subs:test}

We have proved in the previous subsection that there is a continuous map from the low sublevels of the functional $J_\rho$ to the topological join $(\gamma_1)_k*(\gamma_2)_l$, see Proposition \ref{p:map}. Aim of this subsection is to show that we can construct a map in the other way round in a natural way, namely we will introduce a mapping
$$
\Phi_\l: (\gamma_1)_k * (\gamma_2)_l \to J_\rho^{-L},
$$
for large $L$ and the parameter $\l>0$ to be defined in the sequel. We start by taking $\zeta \in (\gamma_1)_k * (\gamma_2)_l, \zeta = (1-s) \s_1 + s \s_2$, with
$$
  \s_1 := \sum_{i=1}^k t_i \d_{x_i} \in (\gamma_1)_k \qquad
  \hbox{ and } \qquad \s_2 :=  \sum_{j=1}^l s_j \d_{y_j} \in
  (\gamma_2)_l.
$$
Our goal is to construct test functions modeled on $\zeta \in (\gamma_1)_k * (\gamma_2)_l$. To this end we set ${\Phi}_\l(\zeta)= \var_{\l , \zeta}$ given by
\begin{equation} \label{bubble}
    \var_{\l , \zeta} (x) =   \log \, \sum_{i=1}^{k} t_i \left( \frac{1}{1 + \l_{1,s}^2 d(x,x_i)^2} \right)^2
    -  \frac 12 \log  \, \sum_{j=1}^{l} s_j \left( \frac{1}{1 + \l_{2,s}^2 d(x,y_j)^2} \right)^2,
\end{equation}
where $\l_{1,s} = (1-s) \l, \l_{2,s} = s \l$. We point out that for $s=1$ the latter definition does not depend on $\s_1$ and similarly for $s=0$ so that it is well-defined with respect to the equivalence relation of the topological join. The main result of this subsection is the following estimate.
\begin{pro} \label{test-f}
Suppose $\rho_1 \in (8 k \pi, 8 (k+1) \pi)$ and $\rho_2 \in (4 l \pi, 4 (l+1) \pi)$.
Then one has
$$
  J_{\rho}(\var_{\l , \zeta}) \to - \infty \quad \hbox{ as  } \l \to + \infty,
  \qquad \quad \hbox{ uniformly in }  \zeta \in (\gamma_1)_k * (\gamma_2)_l.
$$
\end{pro}

\begin{proof}
Let ${v}_1,{v}_2$ be two functions given by
\begin{eqnarray*}
    {v}_1(x) = \log \, \sum_{i=1}^{k} t_i \left( \frac{1}{1 + \l_{1,s}^2 d(x,x_i)^2} \right)^2, \qquad
    {v}_2(x) = \log \, \sum_{j=1}^{l} s_j \left( \frac{1}{1 + \l_{2,s}^2 d(x,y_j)^2} \right)^2,
\end{eqnarray*}
so that $\var = {v}_1 - \frac 12 {v}_2$. We start by considering the part involving the gradient terms, i.e.
\begin{equation} \label{gr}
    \frac 12 \int_{M} |\n \var|^2 \,dV_g  =  \frac 12 \int_M \left(|\n {v}_1|^2 + \frac 14 |\n {v}_2|^2 - \n {v}_1 \cdot \n {v}_2\right) \,dV_g.
\end{equation}
We need here two estimates on the gradients of $v_1$ and $v_2$:
\begin{equation} \label{gr1}
    |\n v_i(x)| \leq C \l_{i,s}, \qquad \mbox{for every $x\in M$ and $s\in[0,1],$} \quad i=1,2,
\end{equation}
where $C$ is a constant independent of $\l$, $\zeta \in (\gamma_1)_k * (\gamma_2)_l$, and
\begin{equation} \label{gr2}
    |\n v_i(x)| \leq \frac{4}{d_{\,i,min}(x)}, \qquad \mbox{for every $x\in M,$} \quad i=1,2,
\end{equation}
where $\dis{d_{1,min}(x) = \min_{i=1,\dots,k} d(x,x_i)}$ and $\dis{d_{2,min}(x) = \min_{j=1,\dots,l} d(x,y_j).}$

\medskip

We prove the inequalities for $v_1$; for $v_2$ we can argue in the same way. We have that
$$
    \n v_1(x) = - 2 \l_{1,s}^2 \frac{\sum_{i=1}^k t_i \bigr(1 + \l_{1,s}^2 d^2(x,x_i)\bigr)^{-3} \n \bigr(d^2(x,x_i)\bigr)}{\sum_{j=1}^k t_j \bigr(1 + \l_{1,s}^2 d^2(x,x_j)\bigr)^{-2}}.
$$
Using $\left|\n \bigr(d^2(x,x_i)\bigr)\right| \leq 2 d(x,x_i)$ jointly with
$$
    \frac{\l_{1,s}^2 d(x,x_i)}{1 + \l_{1,s}^2 d^2(x,x_i)} \leq C \l_{1,s}, \qquad i = 1, \dots, k,
$$
with $C$ a fixed constant, we obtain (\ref{gr1}). We prove now (\ref{gr2}). Observe that if $\l_{1,s} = 0$ the inequality is satisfied. If $\l_{1,s} > 0$ we have
\begin{eqnarray*}
  |\n v_1(x)| & \leq & 4 \l_{1,s}^2 \frac{\sum_{i=1}^k t_i \bigr(1 + \l_{1,s}^2 d^2(x,x_i)\bigr)^{-3} d(x,x_i)}{\sum_{j=1}^k t_j \bigr(1 + \l_{1,s}^2 d^2(x,x_j)\bigr)^{-2}} \leq 4 \l_{1,s}^2 \frac{\sum_{i=1}^k t_i \bigr(1 + \l_{1,s}^2 d^2(x,x_i)\bigr)^{-2} \frac{d(x,x_i)}{\l_{1,s}^2 d^2(x,x_i)}}{\sum_{j=1}^k t_j \bigr(1 + \l_{1,s}^2 d^2(x,x_y)\bigr)^{-2}} \\
   & \leq & 4 \frac{\sum_{i=1}^k t_i \bigr(1 + \l_{1,s}^2 d^2(x,x_i)\bigr)^{-2} \frac{1}{d_{\,1,min}(x)}}{\sum_{j=1}^k t_j \bigr(1 + \l_{1,s}^2 d^2(x,x_j)\bigr)^{-2}} = \frac{4}{d_{\,1,min}(x)},
\end{eqnarray*}
which is (\ref{gr2}).

\medskip

Concerning \eqref{gr} we claim that there exist $C$ depending only on the surface such that
\begin{equation} \label{cl}
\left|\int_M \n {v}_1 \cdot\n {v}_2 \,dV_g\right| \leq C.
\end{equation}
In fact, considering the sets
\begin{equation}\label{sets}
    A_i = \left\{ x \in M : d(x,x_i) = \min_{j=1,\dots,k} d(x,x_j) \right\},
\end{equation}
by (\ref{gr2}) we have
\begin{eqnarray*}
    \int_M \n v_1 \cdot\n v_2 \,dV_g & \leq & \int_M |\n v_1| |\n v_2| \,dV_g \leq 16 \int_M \frac{1}{d_{1,min}(x) \, d_{2,min}(x)} \,dV_g(x) \\
                                       & \leq & 16 \sum_{i=1}^k \int_{A_i} \frac{1}{d(x,x_i) \, d_{2,min}(x)} \,dV_g(x).
\end{eqnarray*}
We divide $A_i$ into $A_i = B_\d(x_i) \cup (A_i \setminus B_\d (x_i)), i = 1,\dots k$, where $\d > 0$ is such that
$$
    \d = \frac 12 \min \left\{ \min_{i\in \{1,\dots k\}, j\in \{1,\dots l\}} d(x_i,
    y_j),\ \
    \min_{m,\, n\in \{1,\dots k\}, m \neq n} d(x_m,x_n) \right\}.
$$
Using a change of variables and observing that $d_{2,min}(x) \geq \frac 1C$ in $B_\d(x_i)$ we obtain
$$
    \sum_{i=1}^k \int_{B_\d(x_i)} \frac{1}{d(x,x_i) \, d_{2,min}(x)} \,dV_g(x) \leq C.
$$
One can carry out the same argument for $A_i \setminus B_\d (x_i)$ by exchanging the role of $d_{1,min}$ and $d_{2,min}$. This proves the claim \eqref{cl}.

\medskip

Suppose that $\l_{1,s}\geq1$ and consider the following splitting
$$
    \frac{1}{2} \int_{M} |\n v_1(x)|^2 \,dV_g(x) = \frac{1}{2} \int_{\bigcup_i B_{\frac{1}{\l_{1,s}}}(x_i)} |\n v_1(x)|^2 \,dV_g(x) + \frac{1}{2}\int_{M \setminus \bigcup_i B_{\frac{1}{\l_{1,s}}}(x_i)} |\n v_1(x)|^2 \,dV_g(x).
$$
From (\ref{gr1}) we have the bound
$$
    \int_{\bigcup_i B_{\frac{1}{\l_{1,s}}}(x_i)} |\n v_1(x)|^2 \,dV_g(x) \leq C.
$$
Using again the sets $A_i$ introduced in (\ref{sets}), by (\ref{gr2}) we deduce
\begin{eqnarray*}
    \frac 12 \int_{M \setminus \bigcup_i B_{\frac{1}{\l_{1,s}}}(x_i)} |\n v_1(x)|^2 \,dV_g  & \leq & 8 \int_{M \setminus \bigcup_i B_{\frac{1}{\l_{1,s}}}(x_i)} \frac{1}{d_{1,min}^2(x)} \,dV_g(x) + C \\
    & \leq & 8 \sum_{i=1}^k \int_{A_i \setminus B_{\frac{1}{\l_{1,s}}}(x_i)} \frac{1}{d_{1,min}^2(x)} \,dV_g(x) + C \\
                                         & \leq & 16 k \pi \log \l_{1,s} + C.
\end{eqnarray*}
Similar estimate holds for $v_2$ and hence recalling \eqref{cl}, by \eqref{gr} we get
\begin{equation} \label{grad}
    \frac 12 \int_{M} |\n \var|^2 \,dV_g  \leq  \frac 12 \int_M \left(|\n {v}_1|^2 + \frac 14 |\n {v}_2|^2\right) \,dV_g \leq 16 k \pi \log (\l_{1,s} +\d_{1,s}) + 4 l \pi \log (\l_{1,s}+\d_{2,s}) + C,
\end{equation}
uniformly in $s\in[0,1]$, for some $\d_{1,s} > \d > 0$ as $s \to 1$ and $\d_{2,s} > \d > 0$ as $s \to 0$, for fixed $\d$.

\medskip

We consider now the nonlinear term. By the definition we have
$$
    \int_M e^{\var} \,dV_g = \sum_{i=1}^k t_i \int_M \frac{1}{\bigr( 1 + \l_{1,s}^2 d(x,x_i)^2 \bigr)^2} \left( \sum_{j=1}^l s_j \frac{1}{\bigr( 1 + \l_{2,s}^2 d(x,y_j)^2 \bigr)^2} \right)^{-\frac 12}\,dV_g(x).
$$
Taking a $\ov{x}\in \{ x_1,\dots x_k \}$ the estimate of the latter integral will be the same of
$$
\int_M \frac{1}{\bigr( 1 + \l_{1,s}^2 d(x,\ov{x})^2 \bigr)^2} \left( \sum_{j=1}^l s_j \frac{1}{\bigr( 1 + \l_{2,s}^2 d(x,y_j)^2 \bigr)^2} \right)^{-\frac 12}\,dV_g(x).
$$
Let $\Sg = B_\d(\ov{x}) \cup (\Sg \setminus B_\d(\ov{x}))$ with ${\d = \frac{\min_j \{ d(\ov{x},y_j) \}}2}$. In $B_\d(\ov{x})$, by a change of variables for the part concerning $\l_{1,s}$ and observing that $\frac 1C \leq d(x, y_j) \leq C, j=1,\dots,l$, for every $x\in B_\d(\ov{x})$, one can conclude
$$
\int_{B_\d(\ov{x})} \frac{1}{\bigr( 1 + \l_{1,s}^2 d(x,\ov{x})^2 \bigr)^2} \left( \sum_{j=1}^l s_j \frac{1}{\bigr( 1 + \l_{2,s}^2 d(x,y_j)^2 \bigr)^2} \right)^{-\frac 12}\,dV_g(x) = \frac{\bigr(\l_{2,s} + \d_{2,s}\bigr)^2}{\bigr(\l_{1,s} + \d_{1,s}\bigr)^2}\bigr(1 + O(1)\bigr).
$$
Observing that in $\Sg \setminus
B_\d(\ov{x})$ it holds $\frac 1C \leq d(x, \ov{x}) \leq C$, it is easy to show that the contribution in this region is negligible with respect to the latter term. Therefore, we get
\begin{equation} \label{exp1}
\log \int_M  e^{\var} \,dV_g = 2 \log \bigr(\l_{2,s} + \d_{2,s}\bigr) - 2 \log \bigr(\l_{1,s} + \d_{1,s}\bigr) + O(1).
\end{equation}
Similar arguments yield
\begin{equation} \label{exp2}
\log \int_M e^{-2\var} \,dV_g = 8\log \bigr(\l_{1,s} + \d_{1,s}\bigr) - 2\log \bigr(\l_{2,s} + \d_{2,s}\bigr) + O(1).
\end{equation}

\medskip

We are left with the average part. For simplicity we estimate just $v_1$ for $k=1$ since for the general case the argument is the same. We claim that
$$
	\int_M v_1 \, dV_g = -4 \log \bigr(\l_{1,s}+\d_{1,s}\bigr) + O(1).
$$
We start by writing (recall we are assuming $k=1$)
$$
	v_1(x) = -4 \log \bigr( \max\{ 1,\l_{1,s}d(x, x_1) \} \bigr) + O(1), \qquad x_1 \in M.
$$
Suppose $\l_{1,s}\geq 1$: then we have
\begin{align*}
	\int_M v_1 \, dV_g &= -4 \int_{M \setminus B_{\frac{1}{\l_{1,s}}}(x_1)} \log \bigr(\l_{1,s}d(x, x_1)\bigr) \,dV_g - 4 \int_{B_{\frac{1}{\l_{1,s}}}(x_1)} dV_g + O(1) \\
	&= -4 \log(\l_{1,s}) \left|M \setminus B_{\frac{1}{\l_{1,s}}}(x_1)\right| -4 \int_{M \setminus B_{\frac{1}{\l_{1,s}}}(x_1)} \log (d(x, x_1)) \,dV_g + O(1).
\end{align*}
Recalling that we set $|M|=1$ the claim holds true. By the definition of $\var$ in \eqref{bubble} we conclude that
\begin{equation} \label{average}
	\int_M \var \, dV_g = -4 \log \bigr(\l_{1,s}+\d_{1,s}\bigr) + 2 \log \bigr(\l_{2,s}+\d_{2,s}\bigr) + O(1).
\end{equation}

\medskip

Using jointly \eqref{grad}, \eqref{exp1}, \eqref{exp2} and \eqref{average} we deduce
$$
  J_{\rho}(\var_{\l , \zeta}) \leq \bigr( 16 k \pi - 2 \rho_1 \bigr)\log \bigr(\l_{1,s} + \d_{1,s}\bigr) + \bigr( 4 l \pi -  \rho_2 \bigr)\log \bigr(\l_{2,s} + \d_{2,s}\bigr) + O(1).
$$
Since $\rho_1 > 8k\pi, \rho_2 > 4l\pi$ and $\dis{\max_{s\in[0,1]}\{ \l_{1,s}, \l_{2,s} \}} \to +\infty$ as $\l \to \infty$, the proof is concluded.
\end{proof}

\medskip

\subsection{The conclusion} \label{subs:conclusion}

In the previous subsections we introduced two maps
\begin{equation} \label{comp}
(\g_1)_k * (\g_2)_l \quad \stackrel{\Phi_{\l}}{\longrightarrow}\quad
J_{\rho}^{-L} \quad\stackrel{\Psi}{\longrightarrow}\quad (\g_1)_k * (\g_2)_l,
\end{equation}
see Propositions \ref{p:map} and \ref{test-f}. We show now a crucial fact, namely that their composition is homotopically equivalent to the identity map on $(\g_1)_k * (\g_2)_l$. Take $\zeta=(1-s) \s_1 + s \s_2 \in (\g_1)_k * (\g_2)_l$, with
$$
\s_1= \sum_{i=1}^k t_i \d_{x_i},\qquad \ \ \s_2= \sum_{j=1}^l s_j\d_{y_j}
$$
and consider $\Phi_\l(\zeta)=\var_{\l,\zeta}$ as in \eqref{bubble}. It is standard to see that
\begin{equation} \label{conv}
\displaystyle \frac{ h_1 e^{\varphi_{\l,\zeta}}}{\int_{M} h_1 e^{\varphi_{\l,\zeta}}\,dV_g} \rightharpoonup \s_1, \qquad \displaystyle \frac{ h_2 e^{-2\varphi_{\l,\zeta}}}{\int_{M} h_2 e^{-2\varphi_{\l,\zeta}}\,dV_g} \rightharpoonup \s_2,
\end{equation}
see for example \cite{bjmr}.

\medskip

Letting $\bar{\zeta}_{\l}= \Psi \circ \Phi_\l (\zeta)=
\bigr(\bar{\s}_{1,\l}, \bar{\s}_{2,\l}, \bar{s}\bigr)\in (\g_1)_k * (\g_2)_l$ the desired homotopy will be given in two: we start by letting $\l\to+\infty$ to get the above convergence and then we pass from $\bar s$ to $s$. The homotopy will be the concatenation of the following maps:
$$
H_i: \bigr ( (\g_1)_k * (\g_2)_l \bigr ) \times [0,1] \to (\g_1)_k *
(\g_2)_l , \qquad i=1,2,
$$
\begin{align*}
&H_1\bigr((\s_1,\s_2,s),\mu\bigr)=
\left(\bar{\s}_{1,\frac{\l}\mu},
\bar{\s}_{2, \frac{\l}\mu},\bar s\right), \\
&H_2\bigr((\s_1,\s_2,s),\mu\bigr)=
\bigr({\s}_{1},
{\s}_{2},(1-\mu)\bar s + \mu s\bigr).
\end{align*}
Observe that $H_1(\cdot,1) =\Psi \circ \Phi_{\l}$. Concerning the first step, for $\l$ fixed and $\mu\to 0$ we get the convergence in \eqref{conv}, hence by
Proposition \ref{l:projbar} we deduce $\displaystyle\psi_k\left(\frac{ h_1 e^{\varphi_{\l,\zeta}}}{\int_{M} h_1 e^{\varphi_{\l,\zeta}}\,dV_g}\right) \to \s_1$, $\displaystyle\psi_l\left(\frac{ h_2 e^{-2\varphi_{\l,\zeta}}}{\int_{M} h_2 e^{-2\varphi_{\l,\zeta}}\,dV_g}\right) \to \s_2$.
Since $\Pi_i$ are retractions, the latter convergence is preserved and we conclude that $\lim_{\mu \to 0} \bar{\s}_{i,\frac \l \mu} \to \s_i$, see the definition of $\Psi$ in \eqref{eq:psi}.

\medskip

We are now in a position to prove the main result of this section.

\medskip

\no \emph{Proof of Theorem \ref{th:existence}.} The equivalence to the identity of the composition of the maps in \eqref{comp} readily implies the following immersion of the homology groups:
\begin{equation} \label{homo}
	H_q\bigr( (\g_1)_k * (\g_2)_l \bigr) \hookrightarrow H_q \left(J_{\rho}^{-L}\right).
\end{equation}
Observe that since each $\gamma_i$ is homeomorphic to $S^1$, $(\g_1)_k$ is homeomorphic to $S^{2k-1}$ while $(\g_2)_l$ to $S^{2l-1}$. Then, $(\g_1)_k * (\g_2)_l$ is homeomorphic to $S^{2k+2l-1}$, see \cite{bjmr} for the references about this arguments. It follows that $(\g_1)_k * (\g_2)_l$ and $J_{\rho}^{-L}$ (by \eqref{homo}) have non-trivial topology.

Suppose by contradiction that \eqref{eq3} has no solutions. Since we are assuming $\rho_1\notin 8\pi\N$, $\rho_2\notin 4\pi\N$, we can apply Corollary \ref{top-arg} to deduce that $J_\rho^{-L}$ is a deformation retract of $J_\rho^{L}$ for any $L>0$. But Corollary~\ref{high} asserts that $J_\rho^{L}$ is contractible for some $L$, which leads $J_\rho^{-L}$ to be contractible too. Hence we get the contradiction.
\begin{flushright}
$\Box$
\end{flushright}

\

\end{document}